\documentclass[12pt]{amsart}
\usepackage{amscd}
\usepackage{graphicx}
\usepackage[breaklinks=true, backref=page]{hyperref}

\usepackage[margin=2.5cm]{geometry}

\newtheorem{thm}{Theorem}[section]

\newtheorem{lma}[thm]{Lemma}

\theoremstyle{definition}

\newtheorem{rmk}[thm]{Remark}

\newenvironment{pf}{\begin{proof}}{\end{proof}}

\numberwithin{equation}{section}

\newcommand{\R}{{\mathbb{R}}}
\newcommand{\C}{{\mathbb{C}}}

\newcommand{\act}{{\mathfrak{a}}}

\newcommand{\Ordo}{{\mathcal{O}}}
\newcommand{\M}{{\mathcal{M}}}

\newcommand{\pa}{\partial}

\newcommand{\area}{\operatorname{area}}
\newcommand{\krn}{\operatorname{ker}}

\newcommand{\sblv}{\mathcal{H}}

\newcommand{\pr}{\operatorname{pr}}

\newcommand{\CZ}{\operatorname{CZ}}

\title[Holomorphic curves for Legendrian surgery\dots]
{Holomorphic curves for Legendrian surgery}
\author{Tobias Ekholm}
\address{Department of mathematics, Uppsala University, 751 06 Uppsala, Sweden \newline
\indent Institut Mittag-Leffler, 182 60 Djursholm, Sweden}
\thanks{TE was supported by the Knut and Alice Wallenberg Foundation and the Swedish Reserach Council}
\subjclass[2010]{53D42, 53D40}
\begin{document}

\begin{abstract}
Let $X$ be a Weinstein manifold with ideal contact boundary $Y$. If $\Lambda\subset Y$ is a link of Legendrian spheres in $Y$ then by attaching Weinstein handles to $X$ along $\Lambda$ we get a Weinstein cobordism $X_{\Lambda}$ with a collection of Lagrangian co-core disks $C$ corresponding to $\Lambda$. In \cite{BEE, EL} it was shown that the wrapped Floer cohomology $CW^{\ast}(C)$ of $C$ in the Weinstein manifold $X'_{\Lambda}=X\cup X_{\Lambda}$is naturally isomorphic to the Legendrian differential graded algebra $CE^{\ast}(\Lambda)$ of $\Lambda$ in $Y$. The argument uses properties of moduli spaces of holomorphic curves, the proofs of which were only sketched. The purpose of this paper is to provide proofs of these properties. 
\end{abstract}

\maketitle
\section{Introduction}\label{S:intro}
A Weinstein $2n$-manifold $X$, see \cite{CE} for terminology, admits a handle decomposition with isotropic handles of dimension $\le n$. It follows from the h-principle for isotropic submanifolds of contact manifolds below maximal dimension that if $X$ is subcritical (i.e., if it admits a handle presentation without $n$-dimensional or \emph{critical} handles) that $X$ is flexible in the sense that any symplectic tangential homotopy equivalence of two subcritical manifolds is homotopic to a symplectomorphism. The class of flexible Weinstein manifolds includes also a class of critical manifolds: those with critical Lagrangian handles attached along a loose Legendrian link, see \cite{EM}. 

Flexibility no longer holds in the presence of critical handles attached along non-loose Legendrian links, which is when symplectic rigidity phenomena come to life. The main tool for detecting rigidity is holomorphic curve theories. For Weinstein manifolds the most basic such theory is the wrapped Floer cohomology of co-core disks $CW^{\ast}(C)$. In \cite{BEE,EL} $CW^{\ast}(C)$ was computed in terms of the Legendrian attaching spheres. 

To state this result with more precision, let $X$ be a Weinstein manifold with contact boundary $Y$ and $\Lambda\subset Y$ a link of Legendrian spheres. Let $X_{\Lambda}$ denote the Weinstein cobordism which is a neighborhood of Lagrangian handles attached along $\Lambda$ and $X'_{\Lambda}=X\cup X_{\Lambda}$. In each handle attached to a component of $\Lambda$ there is a Lagrangian co-core disk. Let $C$ denote the union of all co-core disks. Then $CW^{\ast}(C)$ is isomorphic to the Legendrian differential graded algebra $CE^{\ast}(\Lambda)$ of the attaching spheres in $Y$. 

Many other theories, e.g., the symplectic homology $SH^{\ast}(X_{\Lambda}')$ can be obtained algebraically from $CW^{\ast}(C)$. In \cite{BEE, BEE2} corresponding geometric surgery isomorphisms were described. The key points in the proofs of the surgery formulas are certain geometric properties of Reeb dynamics and moduli spaces of holomorphic disks and spheres. The purpose of this paper is to provide proofs of these. This will then complete the proofs in \cite{BEE, BEE2, EL}. 

We next give a more detailed description of the contents of the paper. In Section \ref{sec:anchor} we study moduli spaces needed for what was called anchoring in \cite{BEE}. Anchoring is useful not only for Legendrian surgery but for defining Legendrian differential graded algebras in fillable contact manifolds and for defining  wrapped Floer cohomology without Hamiltonian. Here one studies holomorphic disks with Lagrangian boundary condition in a Weinstein cobordism with negative end and the desired compactness for moduli spaces used to show chain map equations or that differentials square to zero fails because of splitting off of holomorphic planes at Reeb obits in the negative end. Anchoring then means counting all broken curves with holomorphic planes in the filling of the negative end to retain compactness. 

The holomorphic disks in these theories have a distinguished boundary puncture that induces an asymptotic marker on any orbit where a plane splits off. We utilize this marker to remove symmetry from holomorphic planes and thereby get a comparatively simple perturbation theory that allows for counts over the integers.         

We give a direct proof of the required transversality result using a specific form of perturbation defined near the initial unperturbed moduli space of holomorphic planes. To state the result, let $X$ be a Weinstein manifold and let $J$ be an almost complex structure compatible with the Weinstein structure of $X$. We assume for simplicity that $c_{1}(X)=0$ and $\pi_{1}(X)=0$. (It is straightforward to use the results proved here for Legendrian surgery in more general cases, see \cite[Section 7.3]{EO} for a discussion). Let $\gamma$ be a Reeb orbit with a marker and let $\mathcal{M}(\gamma)$ denote the moduli space of holomorphic spheres with asymptotic marker and positive puncture at $\gamma$. We use the marker and results on the asymptotics of holomorphic curves from \cite{HWZ96} to stabilize the domains of maps in $\mathcal{M}(\gamma)$ and then to define a functional analytic neighborhood $\mathcal{U}$ of $\mathcal{M}$. We show how to construct a  perturbation $\lambda$ of the Cauchy-Riemann equation $\bar\partial_{J}u=0$ to $\bar\pa_{J_{\lambda}} u=0$, where $J_{\lambda}$ is a domain dependent almost complex structure that is allowed to depend not only on the domain but also on the map in $\mathcal{U}$. Let $\mathcal{M}^{\lambda}(\gamma)$ denote the corresponding solution space. Let $|\gamma|=\CZ(\gamma)+(n-3)$ denote the grading of $\gamma$. Here $\CZ$ denotes the Conley-Zehnder index. We prove the following result.
\begin{thm}\label{t:anchor}
For generic perturbation $\lambda$, $\mathcal{M}^{\lambda}(\gamma)$ is a transversely cut out manifold of dimension $|\gamma|$ with a natural compactification $\overline{\mathcal{M}}^{\lambda}(\gamma)$ that consists of several level curves and $\overline{\mathcal{M}}^{\lambda}(\gamma)$ has the structure of a compact orientable manifold with boundary with corners.
\end{thm}

In Sections \ref{S:localmodel} and \ref{S:cobconstr} we give detailed geometric models of Lagrangian handle attachment along a Legendrian sphere and in Section \ref{S:Reebinbasic} we study Reeb dynamics in these models. To state the result, let $Y_{\Lambda}(\epsilon)$ be the contact manifold that results from attaching a width $\epsilon$ handle to the Weinstein manifold $X$ along the Legendrian link $\Lambda$ in its contact boundary $Y$. Let $X_{\Lambda}(\epsilon)$ denote the Weinstein cobordism which is a neighborhood of the handle.  Let $C\subset X_{\Lambda}(\epsilon)$ denote the union of co-core Lagrangian disks and $\Gamma$ the Legendrian boundary of $C$ in $Y_{\Lambda}(\epsilon)$.

\begin{thm}\label{t:Reebdyn}
For Reeb chords of action smaller than $\mathfrak{a}_{\epsilon}$, $\mathfrak{a}_{\epsilon}\to\infty$ as $\epsilon\to 0$, there is a natural 1-1 correspondence between the following sets:
\begin{enumerate}
\item The set of Reeb chords $\mathcal{R}(\Gamma)$ of $\Gamma$ and the set $\Omega(\mathcal{R}(\Lambda))$ of composable words of Reeb chords of $\Lambda$.
\item The set of Reeb orbits $\mathcal{R}^{\circ}(Y_{\Lambda}(\epsilon))$ in $Y_{\Lambda}(\epsilon)$ and the union $\mathcal{R}(Y)\cup \Omega^{\circ}(\mathcal{R}(\Lambda))$, where  $\Omega^{\circ}(\mathcal{R}(\Lambda))$  denotes the set of composable words of Reeb orbits up to cyclic permutation, or simply the set of composable cyclic words.  
\end{enumerate} 
\end{thm} 
See Lemmas \ref{l:newchordsL_0} and \ref{l:newchordsmix} for closely related results for other Legendrians.

In Section \ref{s:buildcob'} we adapt the construction of $X_{\Lambda}$ to holomorphic curves. In Section \ref{A:D} we then construct and count basic holomorphic disks interpolating between the Reeb chords and orbits after surgery and the words of Reeb chords before the surgery in Theorem \ref{t:Reebdyn}. We use the following notation. If $w\in\Omega(\mathcal{R}(\Lambda))$ is a composable word of Reeb chords then $\overline{w}$ denotes the corresponding chord in $\mathcal{R}(c)$ and if $w^{\circ}\in\Omega^{\circ}(\mathcal{R}(\Lambda))$ is a composable cyclic word we write $\overline{w}^{\circ}$ for the corresponding orbit. Note also that if $C$ denotes the co-core disks and $L$ the core disks of the surgery, then $L\cap C$ intersects transversely in one point in each handle attached.

First, consider the unit disk $D$ in the complex plane with punctures fixed at $1$, $i$, and $-1$, and $m$ boundary punctures on the arc between $-1$ and $1$ in the lower half plane. Write $\partial_{+}D$ and $\partial_{-}D$ for the part of the boundary of $D$ that lies in the upper and lower half planes respectively. If $c$ is a Reeb chord of $\Gamma=\partial C$ and $w$ is a word of Reeb chords of $\Lambda=\partial L$ then we write $\mathcal{M}(c,w)$ for the moduli space of holomorphic disks $u\colon(D,\partial D_{+},\partial D_{-})\to (X_{\Lambda},C,L)$ with positive puncture at $c$ and negative punctures in $\partial D_{-}$ at the Reeb chords in the word $w$.   

Second, consider again the unit disk $D$ but now with a fixed puncture at $0$ with an asymptotic marker along the positive real axis and then additional boundary punctures. If $\gamma$ is a Reeb orbit with marker and $w^{\circ}$ is a cyclic word of Reeb chords, we write $\mathcal{M}(\gamma,w^{\circ})$ for the moduli space of holomorphic disks $u\colon (D,\partial D)\to(X_{\Lambda},L)$ with marked positive puncture at the marked Reeb orbit $\gamma$ and negative punctures in $w^{\circ}$.    

\begin{thm}\label{t:count}
Moduli spaces of holomorphic disks interpolating between Reeb chords and orbits after surgery and the corresponding words of Reeb chords before surgery satisfies the following.  	
\begin{enumerate}
\item 
For any word $w\in \Omega(\mathcal{R}(\Lambda))$, 
the moduli space $\mathcal{M}(\overline{w},w)$ is a transversely cut out orientable $0$-manifold with algebraically $\pm 1$ point in it.
\item 
For any composable cyclic word $w^{\circ}\in \Omega(\mathcal{R}^{\circ}(\Lambda))$
and any marker on (the geometric orbit underlying) $\overline{w}^{\circ}$, the moduli space $\mathcal{M}(\overline{w}^{\circ},w^{\circ})$ is a transversely cut out orientable $0$-manifold with algebraically $\pm 1$ point in it.
\end{enumerate} 
\end{thm}
See Lemma \ref{l:gluLeg} for analogous results for other Legendrians.

\section{Moduli spaces for anchoring}\label{sec:anchor}
The purpose of this section is to prove Theorem \ref{t:anchor}. Before going into the proof we elaborate on the motivating discussion from Section \ref{S:intro}.  

\subsection{Background}\label{s:anchorbackground}
Below we give an account of a perturbation scheme needed to study moduli spaces of holomorphic disks in a (possibly trivial) Weinstein cobordism $W$ with negative contact boundary $Y$, where $Y$ is filled by a Weinstein manifold $X$, \emph{anchored} in $X$. The disks have boundary on a Lagrangian submanifold $L\subset W$, boundary punctures at Lagrangian intersections and at Reeb chords in the positive and negative end, and additional interior negative punctures at Reeb orbits. Geometric conditions (one positive puncture, or boundary punctures going to distinct Lagrangians, etc) ensure that the disks can not be multiply covered. Therefore the corresponding moduli spaces are easily seen to be transversely cut out for generic almost complex structure. 

Anchoring concern attaching holomorphic spheres with one positive puncture at the interior negative punctures of the disks. These spheres live in $X$ and we need a description of their compactified moduli spaces. The moduli spaces of punctured spheres are less well behaved, there are curves that branch cover other curves and one must therefore use abstract perturbations to achieve transversality. By now there are a number of standard approaches \cite{HWZ1,HWZ2,HWZ3,Pardon,FOOO1,FOOO2,FOOO3}. It is clear that any of these perturbation schemes can be adapted to the current setup and give the extension of the theories that are natural from the SFT-perspective: chord algebras with coefficients in the orbit contact homology algebra. 

However, for the purposes of proving the Legendrian surgery formula (and to define Chekanov-Eliashberg algebras for fillable contact manifolds) the full strength of such a result is not needed. Rather, it is sufficient to show that the moduli spaces of holomorphic disks with additional interior negative punctures has a good compactification induced by the moduli space of holomorphic spheres in the symplectic filling $X$. In other words we study the moduli space of holomorphic buildings where the lower levels of the building constitute (possibly still several level) punctured spheres that fill the interior negative punctures in higher levels. There are two simplifications here, first we need not preserve any symmetry properties associated to the Reeb orbits at interior punctures. Second, we can take advantage of the geometry of the situation as follows. The top level is a disk with a distinguished boundary puncture. That puncture induces an asymptotic marker at each interior puncture which then allows us to study a compactification where the lower levels consist of holomorphic spheres with one positive puncture with an asymptotic marker. The marker at the positive puncture removes symmetries from the domain and the perturbed moduli spaces become manifolds rather than orbifolds. (From point of view of the orbit algebra, the second part is related to the non-equivariant orbit algebra, see e.g. \cite{EO}.)  

Since it is shorter and much work in this direction has already been carried out in closely related settings, we will here, rather than referring to general existence theorems for perturbations, simply construct a very specific perturbation in a neighborhood of the moduli space (that depends on the maps and domains) and show directly that it has the required properties. More precisely, we use a Cauchy-Riemann equation with almost complex structure that generally depends on both the domain and the map but which outside a neighborhood of broken curves varies only near the puncture and depends only on the angular coordinate sufficiently near the puncture of the holomorphic spheres. To arrange these perturbations we will stabilize the domains of stable holomorphic spheres with unstable domains. Here we use known asymptotic properties of holomorphic planes to find marked points that map to a neighborhood of the Reeb orbits and are singled out by the asymptotic marker. Our argument uses standard tools from Floer theory e.g. Floer gluing and orientations from determinant bundles that have been studied in detail elsewhere, e.g. \cite{BO,EO, ES}. Results from there are easily adapted to the current setting and the corresponding arguments will not be repeated. 

%

\subsection{Asymptotics of holomorphic curves}
Consider the space $\mathcal{J}$ of almost complex structures $J$ on $X$ which are compatible with the symplectic structure $\omega$ and which in the end $[0,\infty)\times Y$ are translation invariant and preserves the contact planes.

Let $u\colon S\to X$ or $u\colon S\to \R\times Y$ be a punctured holomorphic curve. Consider cylindrical coordinates, $s+it\in [0,\infty)\times S^{1}$, in a neighborhood of a puncture in $S$ where $u$ is asymptotic to a Reeb orbit $\gamma$. There are coordinates $(\theta,x,y)\in S^{1}\times T^{\ast}\R^{n-1}$ on a neighborhood of $\gamma$ in $Y$ and a positive function $f\colon S^{1}\times T^{\ast}\R^{n-1}\to \R$ with $f(\theta,0,0)=T$, where $T$ is the action of $\gamma$, and $df(\theta,0,0)=0$, such that the contact form in the neighborhood is
\[ 
\alpha=f(d\theta - ydx).
\]
Let $0>\lambda_{1}>\lambda_{2}>\dots>\lambda_{k}>\dots$ denote the discrete negative eigenvalues of the asymptotic operator 
\[ 
-J_{0}\partial_t z + A_{\gamma}(t) z,\quad A_{\gamma}(t)=J_0(\mathrm{pr}\circ dX\circ \mathrm{pr}),
\]
where $X$ is the Reeb vector field and $\mathrm{pr}$ is projection to the contact plane along $X$. Writing $z=(x,y)$, the following asymptotic result for holomorphic curves $u\colon \Sigma\to X$ or $u\colon \Sigma\to \R\times Y$ with a positive puncture at $\gamma$ is proved in \cite[Theorem 1.4]{HWZ96}: 

Either $u$ is a branched cover of the trivial cylinder $\R\times\gamma$ or there exists $j$ such that 
\begin{equation}\label{eq:fourierexp} 
z(s,t) = e^{\int_{s'}^{s}l_{j}(s)ds}(\phi_{j}(t)+r(s,t)),
\end{equation}
where $l_{j}(s)\to\lambda_{j}<0$, as $s\to\infty$, $r(s,t)\to 0$ uniformly in all derivatives as $s\to\infty$, and where $\phi_{j}$ is a non-vanishing eigenfunction of the eigenvalue $\lambda_{j}$. We say that $u$ with this property has \emph{asymptotics according to the $j^{\rm th}$ eigenvalue}.

With \eqref{eq:fourierexp} established it is not hard to get more refined information about the asymptotics. We have the following result.
\begin{lma}\label{l:morefourier}
	There exists $\delta>0$ and $c_{k}\in\C^{n-1}$ such that 
	\begin{equation}\label{eq:wholefourierexp} 
	z(s,t) \ = \ \left(1+\mathcal{O}(e^{-\delta s})\right)\sum_{k\ge 1} c_{j} e^{\lambda_{k}s}\phi_{j}(t),
	\end{equation}
	where $c_{1}=\dots=c_{j-1}=0$ if $u$ has asymptotics according to the $j^{\rm th}$ eigenvalue.
\end{lma}

\begin{proof}
	As shown in \cite[Equation (33)]{HWZ96} the function $z$ satisfies the differential equation
	\[ 
	\bar\partial z + S(s,t) z= 0
	\]
	on the cylinder $[0,\infty)\times S^{1}$ where $\lim_{s\to \infty}S(s,t)=S_{0}(t)$, so that in these coordinates the asymptotic operator is $i\partial_{t}-S_{0}(t)$. We next show that the operators $\bar\partial + S$ and $\bar\partial + S_{0}$ are conjugate. We conjugate by multiplication by $(\mathbf{1}+h)$ where $h$ is a small smooth $(n-1)\times(n-1)$-matrix-valued function and $\mathbf{1}$ the identity matrix. We get the equation
	\[ 
	(\mathbf{1}+h)^{-1}\left((\mathbf{1}+h)\bar\partial u +\bar\partial(\mathbf{1}+h) u + S(\mathbf{1}+h)u\right)=0
	\]
	or equivalently
	\[ 
	\bar\partial u + \left((\mathbf{1}+h)^{-1}\bar\partial(\mathbf{1}+h) + (\mathbf{1}+h)^{-1}S(\mathbf{1}+h)\right)u = 0.
	\]
	Thus, we look for a matrix valued function $h$ such that
	\begin{equation}\label{eq:conjugation} 
	\bar\partial h + (S h - hS_{0}) = S_{0}-S. 
	\end{equation}
	
	By \ref{eq:fourierexp} we have $|S(s,t)-S_{0}(t)|= \mathcal{O}(e^{-\lambda_{j} s})$.
	By Rellich's theorem the $0$-order operator $h\mapsto Sh-hS_0$ is compact on the Sobolev space $H^{2}$ mapping into $H^{1}$ where we use a small positive weight $\delta$. If we use a small negative weight $-\delta$ at the other infinity of the cylinder $\R\times S^{1}$ then the operator has Fredholm index zero. For trivial perturbation there is no kernel and a unique solution. Writing $\|\cdot\|_{(-\delta,\delta)}$ for the weighted norm, this means that there is a constant $C>0$ such that
	\[ 
	\|h\|_{H^{2}(-\delta,\delta)}\le C\|\bar\partial h\|_{H^{1}(-\delta,\delta)}.
	\] 
	This persists for small perturbation. Using a cut-off function at $s_{0}$, the operator $h\mapsto Sh-hS_0$ has norm $\mathcal{O}(e^{-\delta s_{0}})$ and thus we find a unique solution $h$ of \eqref{eq:conjugation} with
	\[ 
	\|h\|_{H^{2}(-\delta,\delta)}=\mathcal{O}(e^{-\delta s_{0}}).
	\]
	
	It follows that our operators are asymptotically conjugate. Solutions of the standard operator are
	\[ 
	\sum_{j} c_{j}e^{\lambda_{j}s}\phi_{j}(t)
	\]
	and we find that solutions of the actual equation are obtained from them by multiplication by $(\mathbf{1}+h)^{-1}$. It follows by taking $s_{0}\to\infty$ above that that the solutions have asymptotics according to 
	\[ 
	\left(1+\mathcal{O}(e^{-\delta s})\right)\sum_{j} c_{j}e^{\lambda_{j}s}\phi_{j}(t)
	\]
	as claimed.
\end{proof}

\subsection{Asymptotics and stabilizing marked points}
Fix an almost complex structure $J\in\mathcal{J}$. Let $\mathcal{M}$ denote one of two types of moduli spaces:
\begin{itemize}
	\item The moduli space of once punctured holomorphic spheres $u\colon S\to X$ with an asymptotic marker at the positive puncture.
	\item The moduli space of holomorphic spheres with one positive and possibly several negative punctures $u\colon S\to \R\times Y$ and an asymptotic marker at the positive puncture, such that the action of the positive orbit is strictly larger than the sum of the actions of the negative orbits.
\end{itemize}

SFT-compactness \cite{BEHWZ} says that $\mathcal{M}$ is generally not compact but has a natural compactification consisting of several level holomorphic buildings. Here a holomorphic building consist of several levels, where each level is a collection of punctured holomorphic spheres, all with one positive puncture with an asymptotic marker. The levels are joined at Reeb orbits and images of markers at matching punctures agree. Write $\overline{\mathcal{M}}$ for the compactification of $\mathcal{M}$ and let $\mathcal{V}\subset \overline{\mathcal{M}}$ be an open neighborhood of the boundary of $\mathcal{M}$. Then $\mathcal{M}\setminus \mathcal{V}$ is compact.

\begin{lma}\label{l:decaycontrol}
	There exists $m=m(\mathcal{V})<\infty$ such that any curve in $\mathcal{M}\setminus \mathcal{V}$ has leading asymptotics according to the $k^{\rm th}$ eigenvalue for $k\le m$.
\end{lma}

\begin{proof}
	This is a straightforward consequence of SFT-compactness. Assume not, then there exists a sequence $u_{j}\in \mathcal{M}\setminus\mathcal{V}$ that by compactness converges to a curve in $u\in \mathcal{M}\setminus\mathcal{V}$ with trivial leading asymptotics. By the asymptotic properties of holomorphic curves \eqref{eq:fourierexp} $u$ contains a branched covered trivial cylinder and therefore lies in the boundary of $\mathcal{M}$. This contradicts $u\notin\mathcal{V}$. The lemma follows.
\end{proof}

Lemma \ref{l:decaycontrol} gives the following decomposition of $\mathcal{M}\setminus\mathcal{V}$. Let $\mathcal{M}^{k}_{\mathcal{V}}$ denote the subset of holomorphic curves in $\mathcal{M}\setminus\mathcal{V}$ with leading asymptotics according to the $j^{\rm th}$ eigenvalue for $j\ge k$. Then $\mathcal{M}^{k}_{\mathcal{V}}$ is a closed subset for each $k$ and we have
\[ 
\mathcal{M}\setminus\mathcal{V} \ = \ \mathcal{M}^{1}_{\mathcal{V}}\supset\mathcal{M}^{2}_{\mathcal{V}}
\supset\dots\supset \mathcal{M}^{m}_{\mathcal{V}}.
\]

We next explain how we stabilize domains of punctured holomorphic spheres and cylinders.

Let $\gamma$ be a Reeb orbit with marker and as above, let $\mathcal{M}(\gamma)$ be the moduli space of punctured holomorphic spheres with one positive puncture at $\gamma$ and an asymptotic marker.

Fix a tubular neighborhood $N(\gamma;r)\approx \gamma\times D(r)$ of $\gamma$ in $Y$, where $D(r)$ is the $r$-disk in $\C^{n-1}$. Let $D_{u(1)}(r)$ denote the fiber disk at the image of the marker $u(1)$. Then, by \eqref{eq:fourierexp}, if $u$ has leading asymptotics according to the $k^{\rm th}$ eigenvalue then, for all sufficiently small $r$, the preimage $u^{-1}(D_{u(1)}(r))$ consists of $k$ arcs, exactly one of which is asymptotic to the marker at infinity. We say that fiber disks with such preimage are \emph{in good position with respect to $u$}. 

Let $\mathcal{V}$ be a neighborhood of the boundary of $\mathcal{M}$ as in Lemma \ref{l:decaycontrol}.

\begin{lma}\label{l:opencover}
	There is an open cover 
	\[ 
	\mathcal{M}\setminus\mathcal{V} \ \subset \ \mathcal{U}^{1}\cup\dots\cup \mathcal{U}^{m},
	\]
	and radii $r_{m}>r_{m-1}>\dots>r_{1}>0$ with the following properties. 
	\begin{itemize}
		\item $\mathcal{M}_{\mathcal{V}}^{j}\subset \bigcup_{k\ge j}\mathcal{U}^{k}$ for all $j$
		\item $D_{r_{j}}$ is in good position with respect to all maps in $\mathcal{U}^{j}$.
	\end{itemize}
\end{lma}

\begin{proof}
	We use Lemma \ref{l:morefourier}.
	For $r_{m}$ this is obvious from the asymptotics. For a curve $u$ in $\mathcal{U}_{m-1}\cap \mathcal{U}_{m}$ we note that since $c_{m-1}$ in \eqref{eq:wholefourierexp} is bounded away from zero there is $r_{m-1}< r_{m}$ such that the $\phi_{m-1}$-term  in the expansion dominates inside $D(r_{m-1})$. The argument continues inductively, see Figure \ref{fig:comparison}. 
\end{proof}

\begin{figure}[ht]
	\centering
	\includegraphics[width=.4\linewidth]{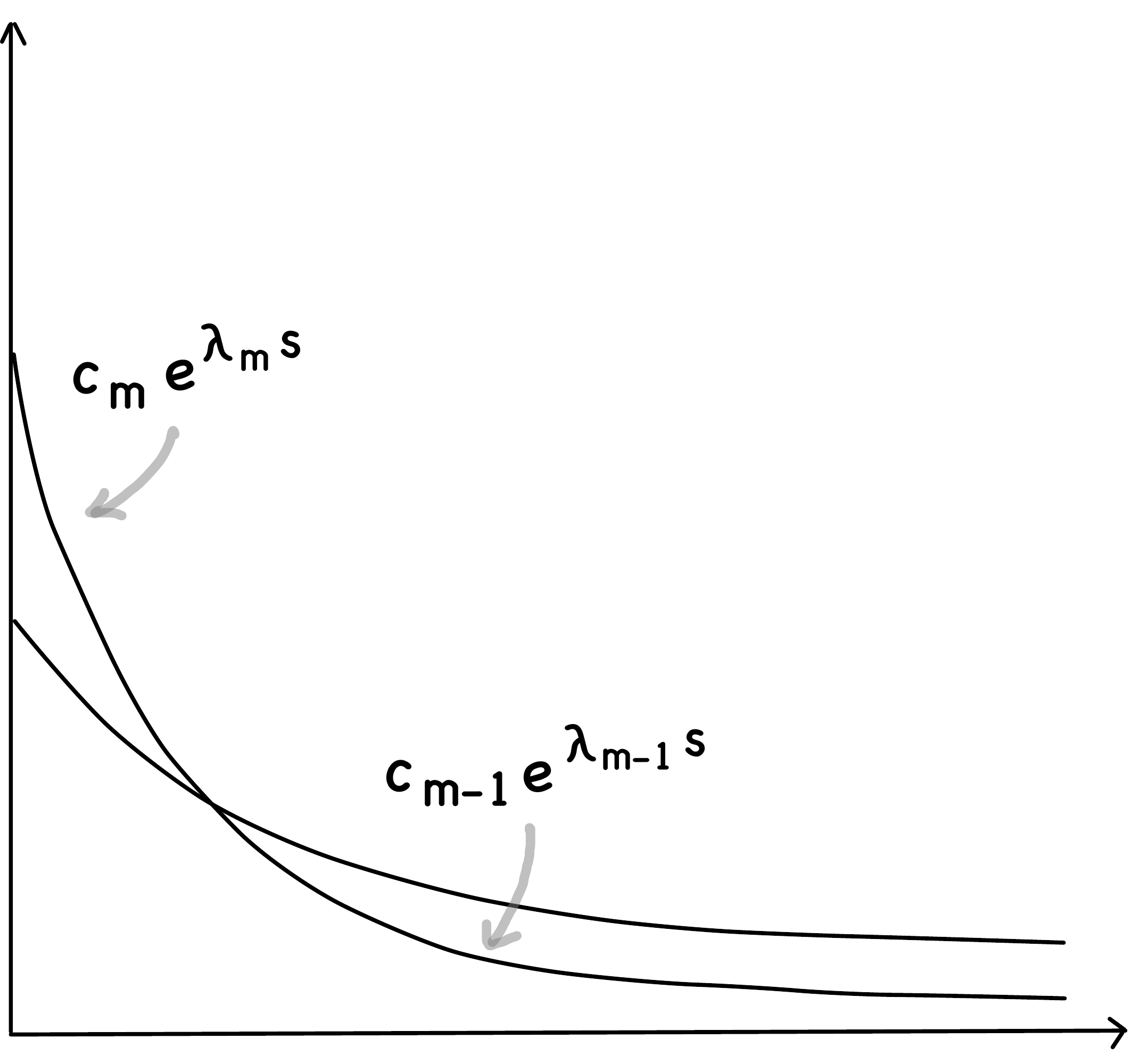}
	\caption{If $c_{m-1}$ is bounded away from zero, the $(m-1)^{\rm th}$ term eventually dominates the $m^{\rm th}$ term.}
	\label{fig:comparison}
\end{figure}

Using the open cover, we define a suitable functional analytic neighborhood of the moduli space. Consider a holomorphic sphere in $\mathcal{U}_{k}$. We fix the domain of $u\colon\C\to X$ or $u\colon\C\to\R\times Y$ as follows. Consider the arc $u^{-1}(D_{u(1)}(r_{k}))$ that is asymptotic to the marker at infinity. Let $\zeta_{0}=u^{-1}\left(\partial D_{u(1)}(r_{k})\right)$ and $\zeta_{1}=u^{-1}\left(\partial D_{u(1)}(r_{k}/2)\right)$. The automorphisms of the domain of $u$ has the form $\zeta\mapsto k(\zeta-b)$, where $k$ is a positive real number and $b$ a complex number. We fix parametrization of $u$ by requiring that $\zeta_{0}=1$ and that $\zeta_{1}$ lies on the circle of radius $2$.

With this parametrization fixed we consider the Sobolev space of maps with two derivatives in $L^{2}$ and small positive exponential weight near the puncture. Here we also require that the maps take the marked points to $D_{u(1)}(r_{k})$ and $D_{u(1)}(r_{k}/2)$, respectively. This gives a functional analytic neighborhood $\mathcal{V}_{k}$ of $\mathcal{U}_{k}$. We need to consider transition maps for $\mathcal{V}_{k}\cap\mathcal{V}_{j}$. Since the closures of $\mathcal{U}_{j}$ are compact it follows that the marked points lie at bounded distance and it is straightforward, using the expansion in Lemma \ref{l:morefourier}, to show that the corresponding coordinate changes on the functional analytic spaces are smooth, compare \cite[Section 4.3]{ES}. 

We parameterize cylinders by the punctured plane $\C^{\ast}$ and use only the first marked point. For curves in the symplectization we use the $\R$-coordinate at the marked point to fix the $\R$-translation.

\subsection{The perturbation scheme}\label{s:pertscheme}
We will use a rather concrete perturbation in the neighborhood of the moduli space constructed above. 

\begin{rmk}
	The resulting perturbation scheme is similar to perturbation data for M-polyfolds (the perturbation depends both on domain and the map). Instead of picking an abstract $\mathrm{sc}^{+}$-section we will rather perturb the equation to an equation with domain dependent almost complex structure and show that this suffices for transversality and gives a perturbed moduli space as a manifold with boundary with corners. 
\end{rmk}

Our construction is inductive in action of the Reeb orbit at the positive puncture and in Hofer energy. 
The first step is trivial:
\subsubsection*{Step 1}
Consider the Reeb orbit $\gamma$ of lowest action. The moduli space of marked holomorphic planes with positive puncture at $\gamma$ is compact. There cannot be any bubbling since bubbles would have smaller action. Furthermore, since the Reeb orbit is simple, any curve with asymptotics at this orbit has an injective point and we get transversality by perturbing $J$. 

\subsubsection*{Inductive step}
Consider an $s$-fold covered Reeb orbit $\gamma^{(s)}$.
We assume inductively that we have found a perturbation scheme for all spheres with positive puncture at Reeb orbits of action less than $\int_{\gamma^{(s)}}\alpha$ and we extend the perturbation scheme to curves with positive puncture at $\gamma^{(s)}$. The perturbation is inductively assumed to have the following properties:
\begin{itemize}
	\item For each perturbed moduli space $\mathcal{M}(\beta)$ there is a $\delta$-neighborhood of the boundary outside of which the perturbed equation agrees with the standard Cauchy-Riemann equation except in cylindrical neighborhoods of the punctures where the equation is domain dependent in such a way that it depends only on the angular coordinate sufficiently close to the puncture. 
	\item In a neighborhood of the boundary the equation is obtained by gluing the equations of the curves in the levels of the building at the boundary.
\end{itemize}   

Consider now the Reeb orbit $\gamma^{(s)}$. Curves of zero action and positive puncture at $\gamma^{(s)}$ are degree $s$ branched covers of the trivial cylinder over $\gamma$, the simple Reeb orbit underlying $\gamma^{(s)}$. To define a perturbation near these branched covers we follow \cite{EO}. 
As in \cite[Section 2.1]{EO}, consider the Deligne-Mumford space of punctured spheres with a distinguished positive puncture and a marker at this puncture. We consider a section of cylindrical ends in the domains and fix a splitting compatible section of complex structures in the contact planes that depends on the $S^{1}$-coordinate in the ends. We turn off the perturbation outside the cylindrical ends. Splitting compatibility means that in a neighborhood of broken curves the almost complex structure depends on the domain in the stretching neck. With these choices we get a domain dependent Cauchy-Riemann equation and the standard argument gives transversality. 

After branched covers of trivial cylinders have been dealt with we consider broken curves at the boundary of the moduli space. By induction and the above construction for branched covers of trivial cylinders, we have compatible perturbation data for all curves in the boundary and by Floer gluing we get solutions to the glued equation for curves in a $\delta$-neighborhood $\mathcal{V}'$ of the boundary of $\mathcal{M}=\mathcal{M}(\gamma^{(s)})$. We point out that all curves glued at this point are transversely cut out: the standard Floer-Picard argument applies. (A much more involved version of the gluing needed in this case can be found in \cite[Section 4.3]{BO}. It is straightforward to adapt he details there to the current case.)

We continue our construction and define the perturbation over all of $\mathcal{M}(\gamma^{(s)})$. We use the cover $\mathcal{U}_{j}$ of Lemma \ref{l:opencover} corresponding to a $\frac14\delta$-neighborhood of the boundary of $\mathcal{M}$. Start at the maximal $j=m$ and pick a domain dependent perturbation for transversality. For curves in $\mathcal{U}_{j-1}\cap \mathcal{U}_{j}$ we then have a perturbation defined. After changing coordinates we extend it to the rest of $\mathcal{U}_{j-1}$. By induction we then have perturbation data over $\mathcal{M}\setminus \mathcal{V}$. Finally, we interpolate between the perturbation data in the $\delta$-neighborhood $\mathcal{V}'$ induced from the boundary perturbation and the perturbation in $\mathcal{V}'\setminus \mathcal{V}$ induced from the perturbation over $\mathcal{M}\setminus\mathcal{V}$ in such a way as to keep the domain dependent almost complex structures near the punctures. For such domain dependent almost complex structures transversality follows from standard arguments by perturbing near the puncture.

\begin{proof}[Proof of Theorem \ref{t:anchor}]
The perturbation scheme described gives a manifold with boundary with corners in a neighborhood of the moduli space in the space of maps (with stabilization as described). By construction, this manifold is transversely cut out and satisfies SFT compactness. The corner structure is inherited from the gluing parameters, compare \cite[Section 6.4 -- 6.6]{ES}. Orientation is induced from the index bundle in the standard way, see e.g., \cite[Appendix: Determinant bundles and signs]{EO}.       
\end{proof}

\section{Local models}\label{S:localmodel}
Our study of symplectic handle attachment and associated cobordisms below are all based on rather detailed local models for handles. In this section we give the details of these local models themselves that are neighborhoods of index $n$ critical points of a Morse function on $\C^{n}$.

\subsection{Notation for vectors and forms}
Consider $\R^{2n}\approx\C^n$ with coordinates $(x,y)=x+iy$, $x,y\in\R^n$. If $u=(u_1,\dots,u_n)$ are coordinates on $\R^n$ then we write
\[
du=(du_1,\dots,du_n)\in((\R^n)^\ast)^n,
\]
and think of it as a vector of cotangent vectors, we also write
\[
\pa_u=(\pa_{u_1},\dots,\pa_{u_n})\in(\R^n)^n,
\]
and think of it as a vector of tangent vectors. If $u,v\in\R^n$ then we write
\[
u\cdot v=\sum_{j=1}^n u_jv_j
\]
and use the abbreviation $u^2=u\cdot u$. Further we denote the norm by $|u|=\sqrt{u^{2}}$. Similarly, if $u\in\R^n$ and $\eta=(\eta_1,\dots,\eta_n)$ is a vector of cotangent- or tangent vectors we write
\[
u\cdot\eta=\sum_{j=1}^n u_j\eta_j,
\]
and finally if $\xi=(\xi_1,\dots,\xi_n)$ and $\eta=(\eta_1,\dots,\eta_n)$ are vectors of cotangent vectors we write
\[
\xi\wedge\eta = \sum_{j=1}^n \xi_j\wedge\eta_j.
\]

Below we will often use a splitting of $\C^{n}=\C\times\C^{n-1}$. We write $x+iy=(x_1+iy_1)+(x_2+iy_2)\in\C\times\C^{n-1}$ and employ the above conventions for vectors and covectors of $\C^{n-1}$.

\subsection{Building blocks for the handle}
In this section, we first describe basic handles which we will use to study Reeb dynamics, then  flattened handles. Flattened handles are small deformations of basic handles that make it easier for us to construct and control certain holomorphic curves.

\subsubsection{Basic handles}
Let $\epsilon>0$, let $p\gg 0$ be an integer, and let $s$ be an integer with $\frac{1}{10}p< s<\frac{1}{5}p$. (Below we will take $\epsilon\to 0$.) Define the hyper-surfaces $V_{\pm\epsilon}\subset\C^n$ as
\begin{equation}\label{e:dfnV}
V_{\pm\epsilon}=\bigl\{(x,y)\in\C^n\colon 2 x_1^2-y_1^2+\epsilon^{2s}(2x_2^{2}-y_2^{2})=\pm\epsilon^{2p}\bigr\}.
\end{equation}
and the handle $H_{\epsilon}\subset \C^{n}$ as the region bounded by $V_{\pm\epsilon}$.
We introduce the following subsets of $\R^{n}\cap W_{\epsilon}$ and of $i\R^{n}\cap W_{\epsilon}$:
\begin{align}\label{eq:corenotation}
D_x(r)&=\bigl\{x\in\R^{n}\colon x_1^{2}+\epsilon^{2s}x_2^{2}\le r^{2}\bigr\},\\\notag
S_x(r)&=\bigl\{x\in\R^{n}\colon x_1^{2}+\epsilon^{2s}x_2^{2}= r^{2}\bigr\},\\\notag
D_y(r)&=\bigl\{y\in i\R^{n}\colon y_1^{2}+\epsilon^{2s}y_2^{2}\le r^{2}\bigr\},\\\notag
S_y(r)&=\bigl\{y\in i\R^{n}\colon y_1^{2}+\epsilon^{2s}y_2^{2}= r^{2}\bigr\}.
\end{align}
Here $D_{y}(\epsilon^{p})$ and $S_{y}(\epsilon^{p})$ are the {\em core disk} and the {\em core sphere} of $H_{\epsilon}$, respectively, and $D_{x}(\epsilon^{p})$ and $S_{x}(\epsilon^{p})$ are the {\em co-core disk} and the {\em co-core sphere} of $H_{\epsilon}$, respectively. See Figure \ref{fig:basic}.

\begin{figure}[ht]
\centering
\includegraphics[width=.55\linewidth]{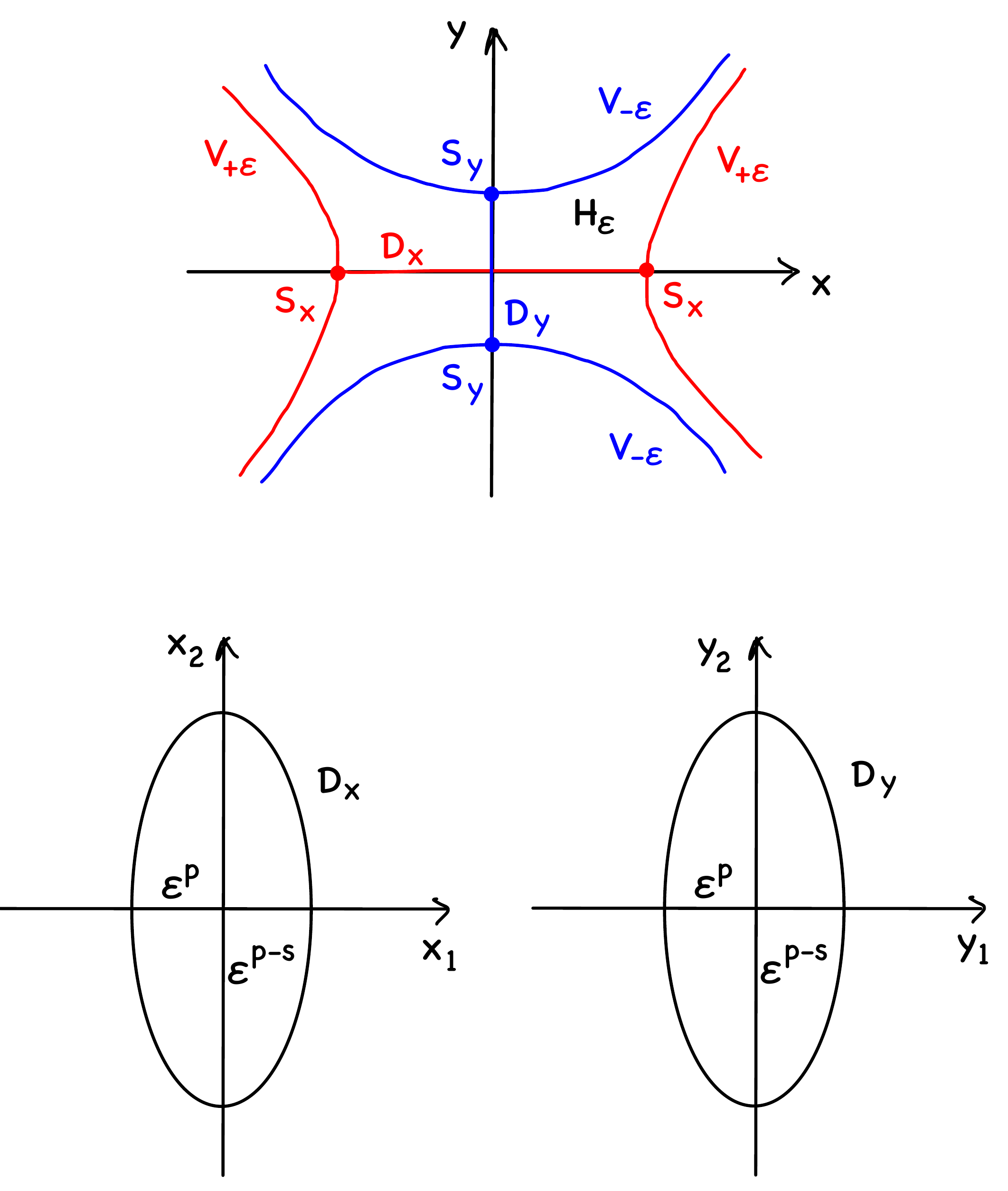}
\caption{The top figure shows the handle $H_{\epsilon}$, its boundary, and the core and co-core disks and spheres. The bottom figure shows more detailed pictures of the core and co-core disks.}
\label{fig:basic}
\end{figure}

\subsubsection{Liouville- and Reeb vector fields}
The vector field
\begin{equation}\label{e:LiouvilleV}
v=2 x\cdot\pa_x-y\cdot\pa_y
\end{equation}
is a Liouville vector field of the standard symplectic form $\omega_{\rm st}=dx\wedge dy$ on $\C^{n}$: the Lie derivative $L_{v}$ in direction $v$ satisfies
\[
L_{v}\omega_{\rm st}=d\bigl(\iota_{v}(dx\wedge dy)\bigr)=\omega_{\rm st}.
\]
Consider the normal vector field
\[
n=2x_1\pa_{x_1}-y_{1}\pa_{y_1}+\epsilon^{2s}(2x_2\cdot\pa_{x_2}-y_{2}\cdot\pa_{y_2})             \]
of $V_{\pm\epsilon}$. Since
\[
n\cdot v=4x_{1}^{2}+y_{1}^{2}+\epsilon^{2s}(4x_{2}^{2}+y_{2}^{2})>0,
\]
$v$ is transverse to $V_{\pm\epsilon}$ and
\[
\alpha=\iota_{v}(\omega_{\rm st})=2x\cdot dy + y\cdot dx
\]
is a contact form on $V_{\pm\epsilon}$. The Reeb vector field of $\alpha$ on $V_{\pm\epsilon}$ is
\begin{equation}\label{e:ReebV}
R_{\alpha}=N(x,y)\left(2x_1\pa_{y_1}+y_1\pa_{x_1} + \epsilon^{2s}(2x_2\cdot\pa_{y_2}+y_2\cdot\pa_{x_2})\right),
\end{equation}
where the normalization factor $N(x,y)$ satisfies
\begin{equation}\label{e:ReebVnormaliz}
(N(x,y))^{-1}=4x_1^{2}+y_1^{2}+\epsilon^{2s}(4x_2^{2}+y_2^{2}).
\end{equation}

\subsection{The Reeb flow on $V_{\pm\epsilon}$}
In this section we describe the Reeb flow on the boundary of basic handles. Changing the speed of the Reeb vector field $R_{\alpha}$ in \eqref{e:ReebV}, its flow is given by the solution to the linear differential equation
\[
\begin{cases}
\dot x_1 = y_1,\\
\dot y_1 = 2x_1,\\
\dot x_2 = \epsilon^{2s}y_2,\\
\dot y_2 = 2\epsilon^{2s}x_2.
\end{cases}
\]
which, for initial condition $(x(0),y(0))$, is
\begin{equation}\label{e:Reebflow}
\begin{cases}
x_1(t) = \cosh(\sqrt{2}t)\, x_1(0) + \frac{1}{\sqrt{2}}\sinh(\sqrt{2}t)\,y_1(0),\\
y_1(t) = \sqrt{2}\sinh(\sqrt{2}t)\, x_1(0) + \cosh(\sqrt{2}t)\, y_1(0),\\
x_2(t) = \cosh(\sqrt{2}\epsilon^{2s} t)\, x_2(0) + \frac{1}{\sqrt{2}}\sinh(\sqrt{2}\epsilon^{2s}t)\,y_2(0),\\
y_2(t) = \sqrt{2}\sinh(\sqrt{2}\epsilon^{2s}t)\, x_2(0) + \cosh(\sqrt{2}\epsilon^{2s}t)\, y_2(0).
\end{cases}
\end{equation}

Consider
\[
\R^{n}\times D_{y}(\epsilon^{s+1})\subset\R^{n}\times i\R^{n}=\C^{n}.
\]
The intersection
\[
N_{+\epsilon}(\epsilon^{s+1})=
(\R^{n}\times D_{y}(\epsilon^{s+1}))\cap V_{+\epsilon}
\]
is a tubular neighborhood of the Legendrian sphere $S_{x}(\tfrac{1}{\sqrt{2}}\epsilon^{p})\subset V_{+\epsilon}$:
\[
N_{+\epsilon}(\epsilon^{s+1})=\bigl\{(x,y)\in V_{+\epsilon}\colon y_1^{2}+\epsilon^{2s}y_2^{2}\le \epsilon^{2s+2}\bigr\}\approx_{\phi_{\epsilon}} S_x(\tfrac{1}{\sqrt{2}}\epsilon^{p})\times D_y(\epsilon^{s+1}),
\]
where the diffeomorphism
\[
\phi_{\epsilon}\colon S_x(\tfrac{1}{\sqrt{2}}\epsilon^{p})\times D_y(\epsilon^{s+1})\to
N_{+\epsilon}(\epsilon^{s+1})
\]
is
\[
\phi(x,y)=
\left(\sqrt{1+\epsilon^{-2p}(y_1^{2}+\epsilon^{2s}y_2^{2})}\,x,\,y\right).
\]
The boundary $\pa N_{+\epsilon}(\epsilon^{s+1})$ is the product
\begin{align*}
\pa N_{+\epsilon}(\epsilon^{s+1})
&=\bigl\{(x,y)\in V_{+\epsilon}\colon y_1^{2}+\epsilon^{2s}y_2^{2}=\epsilon^{2s+2}\bigr\}\\
&=S_x(\tfrac{1}{\sqrt{2}}\epsilon^{s+1}\sqrt{1+\epsilon^{2p-2s-2}})\times S_y(\epsilon^{s+1}).
\end{align*}

We study the Reeb flow through $N_{+\epsilon}(\epsilon^{s+1})$. Fix $y(0)\in S_y(\epsilon^{s+1})$. Denote the time $t$ Reeb flow starting at $(x,y(0))\in \pa N_{+\epsilon}(\epsilon^{s+1})$ by $\Phi^{t}(x)$. Define
\[
\tau(x)=\sup\left\{t\ge 0\colon \pi_{y}(\Phi^{t}(x))\in D_y(\epsilon^{s+1})\right\},
\]
where $\pi_{y}(x,y)=y$. It is easy to see that $\tau(x)$ is uniformly bounded.

\begin{lma}\label{l:flowthrhandle}
The function $\tau(x)$ satisfies $\tau(x)>0$ if and only if $x_1 y_1(0)+\epsilon^{4s}x_2\cdot y_2(0)<0$. The map
\[
x\mapsto\pi_y\bigl(\Phi^{\tau(x)}(x,y(0))\bigr)
\]
is a diffeomorphism from the hemisphere
\[
\left\{x\in S_x(\tfrac{1}{\sqrt{2}}\epsilon^{s+1}\sqrt{1+ \epsilon^{2p-2s-2}})\colon x_1y_1(0)+\epsilon^{4s}x_2\cdot y_2(0)<0\right\}
\]
to the punctured sphere $S_y(\epsilon^{s+1})-\{(y(0))\}$.
\end{lma}

\begin{pf}
The derivative of the flow $(x(t),y(t))$ of \eqref{e:Reebflow} with initial condition $(x,y(0))$ satisfies
\[
\left.\frac{d}{dt}\right|_{t=0} y_1^2+\epsilon^{2s}y_2^{2} = 4 (x_1y_1(0)+\epsilon^{4s}x_2\cdot y_2(0)).
\]
Equation \eqref{e:Reebflow} implies that for any $y\in S_y(\epsilon^{s+1})-\{y(0)\}$ there is a unique $x\in S_x(\tfrac{1}{\sqrt{2}}\epsilon^{s+1}\sqrt{1+ \epsilon^{2p-2s-2}})$ with $x_1y_1(0)+\epsilon^{4s}x_2\cdot y_2(0)<0$ such that $\pi_y\bigl(\Phi^{\tau(x)}(x,y(0))\bigr)=y$ and that the last statement holds, see Figure \ref{fig:flow}.
\end{pf}

\begin{figure}[ht]
\centering
\includegraphics[width=.25\linewidth]{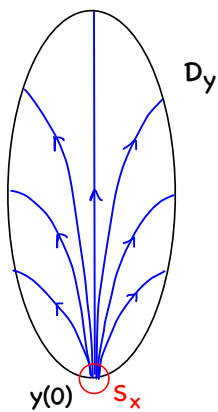}
\caption{The Reeb flow of in $N_{+\epsilon}(\epsilon^{s+1})$ projected to $D_{y}(\epsilon^{s+1})$. The initial condition in $S_{x}(\tfrac{1}{\sqrt{2}}\epsilon^{s+1}\sqrt{1+ \epsilon^{2p-2s-2}})$ is illustrated schematically by the intersection of the flow line and an infinitesimal sphere surrounding the starting point.}
\label{fig:flow}
\end{figure}

\subsection{Neighborhoods of spheres}\label{s:sphnbhd}
Let $\Lambda\subset Y$ be a Legendrian sphere in a contact manifold. Fix a contact identification of a neighborhood of $\Lambda\subset Y$ with a neighborhood of the $0$-section in $J^1(S^{n-1})$. We use coordinates $(q,p,z)\in\R^n\times\R^n\times\R$ on $J^1(S^{n-1})$, where we identify $J^1(S^{n-1})$ with a subset of $\R^n\times\R^n\times\R$ as follows
\[
J^{1}(S^{n-1})=\bigl\{(q,p,z)\in\C^{n}\times\R\colon q^2=1,\quad p\cdot q=0\bigr\}.
\]
We use the standard contact form $\alpha_{\rm st}$ on $J^{1}(S^{n-1})$ given by
\[
\alpha_{\rm st}=dz-p\cdot dq
\]
and with Reeb vector field $R_{\alpha_{\rm st}}=\pa_z$. Furthermore, we consider $T^{\ast}S^{n-1}\subset J^{1}(S^{n-1})$ as
\[
T^{\ast}S^{n-1}=\bigl\{(q,p,z)\in J^{1}(S^{n-1})\colon z=0\bigr\}.
\]
We will consider tubular neighborhoods of the zero section $S^{n-1}\subset J^{1}(S^{n-1})$ of the form
\begin{equation}\label{e:ellnbhd}
\bigl\{(p,q,z)\colon Q_q(p)\le r^{2}, -f_q(Q_q(p))\le z\le f_q(Q_q(p))\bigr\},
\end{equation}
where $Q_q$ are positive definite quadratic forms and where $f_q(s)$ are non-negative functions on $[0,r^{2}]$ which vanishes at $r^{2}$.

Let $\epsilon>0$. Consider
\[
D_{x}(\tfrac{1}{\sqrt{2}}\epsilon^{s+1})\times i\R^{n}\subset\R^{n}\times i\R^{n}=\C^{n}.
\]
The intersection
\[
N_{-\epsilon}(\tfrac{1}{\sqrt{2}}\epsilon^{s+1})=
(D_{x}(\tfrac{1}{\sqrt{2}}\epsilon^{s+1})\times i\R^{n})\cap V_{-\epsilon}
\]
is a tubular neighborhood of the Legendrian sphere $S_{y}(\epsilon^{p})\subset V_{-\epsilon}$:
\begin{align*}
N_{-\epsilon}(\tfrac{1}{\sqrt{2}}\epsilon^{s+1})&=\left\{(x,y)\in V_{-\epsilon}\colon  x_1^{2}+\epsilon^{2s}x_2^{2}\le\tfrac{1}{2}\epsilon^{2s+2}\right\}\\
&\approx_{\phi_{-\epsilon}} S_y(\epsilon^{p})\times D_x(\tfrac{1}{\sqrt{2}}\epsilon^{s+1}),
\end{align*}
where the diffeomorphism
\[
\phi_{-\epsilon}\colon S_y(\epsilon^{p})\times D_x(\tfrac{1}{\sqrt{2}}\epsilon^{s+1})\to
N_{-\epsilon}(\tfrac{1}{\sqrt{2}}\epsilon^{s+1})
\]
is
\[
\phi_{-\epsilon}(x,y)=
\left(x,\sqrt{1+2\epsilon^{-2p}(x_1^{2}+\epsilon^{2s}x_2^{2})}\,y\right)
\]
The boundary $\pa N_{-\epsilon}(\tfrac{1}{\sqrt{2}}\epsilon^{s+1})$ is the product
\begin{align*}
\pa N_{-\epsilon}(\tfrac{1}{\sqrt{2}}\epsilon^{s+1})&=
\left\{(x,y)\in V_{-\epsilon}\colon x_1^{2}+\epsilon^{2s}x_2^{2}=\tfrac12\epsilon^{2s+2}\right\}\\
&=S_x(\tfrac{1}{\sqrt{2}}\epsilon^{s+1})\times S_y(\epsilon^{s+1}\sqrt{1+\epsilon^{2p-2s-2}}).
\end{align*}

Define $W_{-\epsilon}\subset V_{-\epsilon}$ as the subset
\[
W_{-\epsilon}=\{(x,y)\in V_{-\epsilon}\colon x\cdot y=0\}.
\]
Consider the map
\[
\psi\colon W_{-\epsilon}\to T^{\ast}S^{n-1}\subset J^{1}(S^{n-1})
\]
defined by the formula
\begin{equation}
\psi(x,y)=\bigl(q(x,y),p(x,y),z(x,y)\bigr)=\bigl(|y|^{-1} y,\,\, -|y| x,\,\, 0\bigr).
\end{equation}

\begin{lma}\label{l:charts}
Let
\[
B_{-\epsilon}(\tfrac{1}{\sqrt{2}}\epsilon^{s+1})=W_{-\epsilon}\cap N_{-\epsilon}(\tfrac{1}{\sqrt{2}}\epsilon^{s+1}).
\]
Then $B_{-\epsilon}(\tfrac{1}{\sqrt{2}}\epsilon^{s+1})$ is diffeomorphic to a disk sub-bundle of $T^{\ast}S_y(\epsilon^{p})$ and $\psi\colon B_{-\epsilon}(\tfrac{1}{\sqrt{2}}\epsilon^{s+1})\to T^{\ast} S^{n-1}\subset J^{1}(S^{n-1})$ is an embedding
into a neighborhood of the $0$-section of the form
\[
\{(q,p)\in T^{\ast}S^{n-1}\colon Q_q(p)\le 1\},
\]
where $Q_q$ is a family of positive definite quadratic forms such that for each $q$ the surface
$\{p\colon Q_q(p)=1\}$ lies between the spheres $|p|=C_0\epsilon^{2}$ and $|p|=c_0\epsilon^{2s+2}$ for positive constants $c_0$ and $C_0$.
Furthermore,
\[
\psi^{\ast}(\alpha_{\rm st}|T^{\ast}S^{n-1})=\alpha|B_{-\epsilon}(\tfrac{1}{\sqrt{2}}\epsilon^{s+1}).
\]
\end{lma}

\begin{pf}
To see that the first statement holds note that for given $q\in S^{n-1}$, the set of points which map to the fiber $T^{\ast}_qS^{n-1}$ is
\[
\left(D_x(\tfrac{1}{\sqrt{2}}\epsilon^{s+1})\times \R y\right)\cap W_{-\epsilon},
\]
where $\tfrac{y}{|y|}=q$. Projecting out $y$, the intersection projects to a solid ellipsoid in $\R^{n}$ and the image under $\psi$ is that ellipsoid mapped to the fiber over $q$ and scaled by $|y|$. The maximal length of a vector in $S_x(\tfrac{1}{\sqrt{2}}\epsilon^{s+1})$ is of size $\epsilon$ and the minimal length is of size $\epsilon^{s+1}$. The same estimates hold for $|y|$ if $y\in S_y(\epsilon^{s+1}\sqrt{1+\epsilon^{2p-2s-2}})$  and the first statement follows. To prove the last statement, we calculate
\[
\psi_{\pm}^\ast(dq)=\left(|y|^{-1}dy-(y\cdot dy)|y|^{-3} y\right)
\]
and since $x\cdot y=0$ along $W_{-\epsilon}$,
\[
\psi^\ast(p\cdot dq)=x\cdot dy|W_{-\epsilon}=(2x\cdot dy + y\cdot dx)|W_{-\epsilon}=\alpha|W_{-\epsilon}.
\]
\end{pf}

Note that the Reeb field $R_{\alpha}$ is transverse to $W_{-\epsilon}$. We use the Reeb flow with initial condition in $B_{-\epsilon}(\tfrac{1}{\sqrt{2}}\epsilon^{s+1})$ to define a contact embedding of $N_{-\epsilon}(\tfrac{1}{\sqrt{2}}\epsilon^{s+1})$ into $J^{1}(S^{n-1})$. Its image will be a neighborhood of the form \eqref{e:ellnbhd}.

Let $\Phi^{t}(p)$, $p\in V_{-\epsilon}$ denote the time $t$ Reeb flow starting at $p$. Consider the space $B_{-\epsilon}(\tfrac{1}{\sqrt{2}}\epsilon^{s+1})\times(-1,1)$ and consider the map
\[
\Theta\colon B_{-\epsilon}(\tfrac{1}{\sqrt{2}}\epsilon^{s+1})\times(-1,1) \to V_{-\epsilon},\quad
\Theta((x,y),t)=\Phi^{t}((x,y)).
\]
Define
\begin{equation}
\tilde N_{-\epsilon}(\tfrac{1}{\sqrt{2}}\epsilon^{s+1})=
\Theta^{-1}\left(N_{-\epsilon}(\tfrac{1}{\sqrt{2}}\epsilon^{s+1})\right).
\end{equation}
Since Reeb flows preserves contact forms, the contact form $\alpha$ on $V_{-\epsilon}$ is given by
\begin{equation}\label{e:coordcontform}
\alpha^{\rm res}+dt=
(2x\cdot dy + y\cdot dx)^{\rm res} + dt,
\end{equation}
where $t$ is a coordinate on $(-1,1)$ and where $\beta^{\rm res}$ denotes the form $\beta$ restricted to $B_{-\epsilon}(\tfrac{1}{\sqrt{2}}\epsilon^{s+1})$.

We define the map $\Psi\colon \tilde N_{-\epsilon}(\tfrac{1}{\sqrt{2}}\epsilon^{s+1})\to J^{1}(S^{n-1})$ as
\[
\Psi((x,y),t)=(\psi(x,y),t)\in T^{\ast}S^{n-1}\times\R=J^{1}(S^{n-1}),
\]
where $\psi\colon B_{-\epsilon}(\tfrac{1}{\sqrt{2}}\epsilon^{s+1})\to T^{\ast} S^{n-1}$ is the map of Lemma \ref{l:charts}.

\begin{lma}\label{l:ellnbhd}
The map $\Psi$ embeds $N_{-\epsilon}(\tfrac{1}{\sqrt{2}}\epsilon^{s+1})\approx \tilde N_{-\epsilon}(\tfrac{1}{\sqrt{2}}\epsilon^{s+1})$ into $J^{1}(S^{n-1})$. Its image is of the form \eqref{e:ellnbhd} where $Q_q$ is as in Lemma \ref{l:charts} and where $\max_q\max(f_q)=\Ordo(\epsilon^{2s+2s}\log(\epsilon^{1-p}))$. Furthermore, $\Psi^{\ast}(\alpha_{\rm st})=\alpha$, i.e., $\Psi$ is a contact embedding.
\end{lma}

\begin{pf}
The statement on contact forms is immediate from \eqref{e:coordcontform}. To see that the size of the neighborhood is correct, note that given an initial condition in $B_{-\epsilon}(\tfrac{1}{\sqrt{2}}\epsilon^{s+1})$, a crude estimate is that it takes the re-normalized Reeb flow of \eqref{e:Reebflow}  $\Ordo\left(\log(\epsilon^{1-p})\right)$ to reach $\pa N_{-\epsilon}(\tfrac{1}{\sqrt{2}}\epsilon^{s+1})$. In $N_{-\epsilon}(\tfrac{1}{\sqrt{2}}\epsilon^{s+1})$, the minimal ratio of the length of the re-normalized Reeb field and the length of the original Reeb field is larger than $C\epsilon^{-2s-2}$ for some constant $C>0$ and the last statement follows as well.
\end{pf}

\section{Construction of a cobordism}\label{S:cobconstr}
Let $Y$ be a contact $(2n-1)$-manifold and let $\Lambda\subset Y$ be a Legendrian $(n-1)$-sphere. In this section we construct an explicit family of models for the exact symplectic cobordism $X_\Lambda$ with positive contact boundary $Y_{\Lambda}$ and negative contact boundary $Y$ obtained by attaching a handle to $\pa(\Lambda\times[0,\infty))\subset \pa (Y\times(0,\infty])$.

\subsection{The construction of $Y_\Lambda(\epsilon)$}\label{s:constrYLambda}
Fix an identification of a neighborhood $N(\Lambda)$ of $\Lambda\subset Y$ with a neighborhood of the $0$-section in $J^{1}(S^{n-1})$ as in Subsection \ref{s:sphnbhd}. Let $\Psi$ denote the map of Lemma \ref{l:ellnbhd}. Use the map $\Psi$ to identify a smaller neighborhood, of the form described in Lemma \ref{l:ellnbhd}, of $\Lambda\subset Y$ with the neighborhood $N_{-\epsilon}(\tfrac{1}{\sqrt{2}}\epsilon^{s+1})$ of $S_{y}(\epsilon^{p})\subset V_{-\epsilon}$.

\begin{rmk}\label{r:refmetric}
In order to keep track of neighborhood sizes we fix a reference metric on $Y$ which agrees with the standard metric in $J^{1}(S^{n-1})$.
\end{rmk}

The contact manifold $Y_\Lambda(\epsilon)$ which results from surgery on $\Lambda\subset Y$ is identical to $Y$ outside $N(\Lambda)$ and obtained by attaching a neighborhood of $N_{+\epsilon}(\tilde\epsilon^{s+1})\subset V_{+\epsilon}$ along a neighborhood of $\pa N(\Lambda)\approx\pa N_{-\epsilon}(\tfrac{1}{\sqrt{2}}\epsilon^{s+1})$, where $\tilde\epsilon=\epsilon + \Ordo(\epsilon^{l})$ for some 
\begin{equation}\label{eq:paraml}
5s+5<l<p.
\end{equation}
We now give the details of this construction.

Consider the Liouville flow $(x(t),y(t))=\Omega^{t}(x(0),y(0))$, $t\in \R$, of the vector field $v$ on $\C^{n}$, see \eqref{e:LiouvilleV}, with initial condition $(x(0),y(0))$. This flow is given by
\begin{equation}\label{e:Liouvilleflow}
\begin{cases}
x(t)=x(0)\,e^{2t},\\
y(t)=y(0)\,e^{-t}.
\end{cases}
\end{equation}
We will consider initial conditions $(x(0),y(0))\in V_{-\epsilon}$ in
\[
A_{-\epsilon}=
N_{-\epsilon}(\tfrac{1}{\sqrt{2}}\epsilon^{s+1})-N_{-\epsilon}(\tfrac{1}{2\sqrt{2}}\epsilon^{s+1}).
\]
For such initial conditions, define $T(x(0),y(0))$ through the equation
\[
\Omega^{T(x(0),y(0))}(x(0),y(0))\in V_{+\epsilon},
\]
see Figure \ref{fig:attachflow}.

\begin{figure}[ht]
\centering
\includegraphics[width=.55\linewidth]{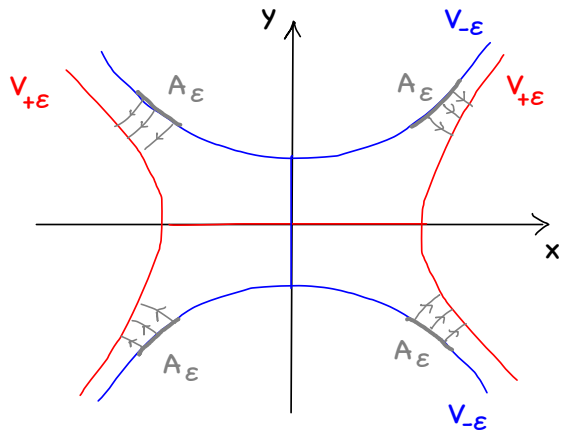}
\caption{The Liouville flow with initial condition in $A_{-\epsilon}$ gives rise to the attaching map for the surgery.}
\label{fig:attachflow}
\end{figure}

\begin{lma}
On $A_{-\epsilon}$, the function $T$ depends only on $x_1(0)^{2}+\epsilon^{2s}x_2^{2}(0)$, i.e.,
\[
T(x(0),y(0))=T(x_1(0)^{2}+\epsilon^{2s}x_2^{2}(0)),
\]
and satisfies
\begin{align}\label{e:Tbound}
T&=\Ordo(\epsilon^{2p-2s-2}),\\\label{e:T'bound}
|d^{(k)}T|&=\Ordo(\epsilon^{2p-(2+k)(s+1)}),\quad k=1,2,3,
\end{align}
where $d^{(k)}T$ denotes the $k^{\rm th}$ derivative of $T$.
\end{lma}

\begin{pf}
Write $T=T(x(0),y(0))$. Then
\[
2x_1(0)^{2}e^{4T}-y_1(0)^{2}e^{-2T}+\epsilon^{2s}(2x_2(0)^{2}e^{4T}-y_2(0)^{2}e^{-2T})=\epsilon^{2p}.
\]
Using $(x(0),y(0))\in V_{+\epsilon}$, we rewrite this as
\[
T\left(\tfrac{e^{4T}-e^{-2T}}{T}\right)=
\frac{\epsilon^{2p}}{x_1(0)^{2}+\epsilon^{2s}x_2(0)^{2}}\left(\tfrac{1+ e^{-2T}}{2}\right).
\]
Since $x_1(0)^{2}+\epsilon^{2s}x_2(0)^{2}\ge \tfrac{1}{8}\epsilon^{2s+2}$, since the left hand side is increasing in $T$, and since the right hand side is decreasing in $T$ we conclude that $T\le C\epsilon^{2p-2s-2}$ for some constant $C$. For such $T$ the functions in brackets are approximately constant and we find that
\[
T(x(0),y(0))=\frac{\epsilon^{2p}}{x_1(0)^{2}+\epsilon^{2s}x_2(0)^{2}}\cdot \phi(x(0)),
\]
where $\phi(x(0))$ is a smooth function.
\end{pf}

Let
\[
T=T(\tfrac{1}{\sqrt{2}}\epsilon^{s+1}-10\epsilon^{l}),
\]
see \eqref{eq:paraml}, denote the flow time it takes for an initial condition $(x(0),y(0))\in V_{-\epsilon}$ with $x_1(0)^{2}+\epsilon^{2s}x_2(0)^{2}=\tfrac{1}{\sqrt{2}}\epsilon^{s+1}-10\epsilon^{l}$ to reach $V_{+\epsilon}$.
Consider the hypersurface $E'\subset V_{-\epsilon}$,
\[
E'=\bigl\{(x,y)\in V_{-\epsilon}\colon \tfrac{1}{\sqrt{2}}\epsilon^{s+1}-5\epsilon^{l}\le
x_1^{2}+\epsilon^{2s}x_2^{2}\le \tfrac{1}{\sqrt{2}}\epsilon^{s+1}\bigr\}.
\]
and let
\[
E=\Omega^{T}(E').
\]
By Lemma \ref{e:T'bound}, the distance from $E$ to $V_{+\epsilon}$ is $\Ordo(\epsilon^{2p-3s-2+l})$.

Let
\begin{equation}\label{e:N'}
N'_\epsilon(\tilde\epsilon^{s+1})\subset\C^{n}
\end{equation}
denote a hyper-surface with the following properties, see Figure \ref{fig:interpol}.
\begin{figure}[ht]
\centering
\includegraphics[width=.4\linewidth]{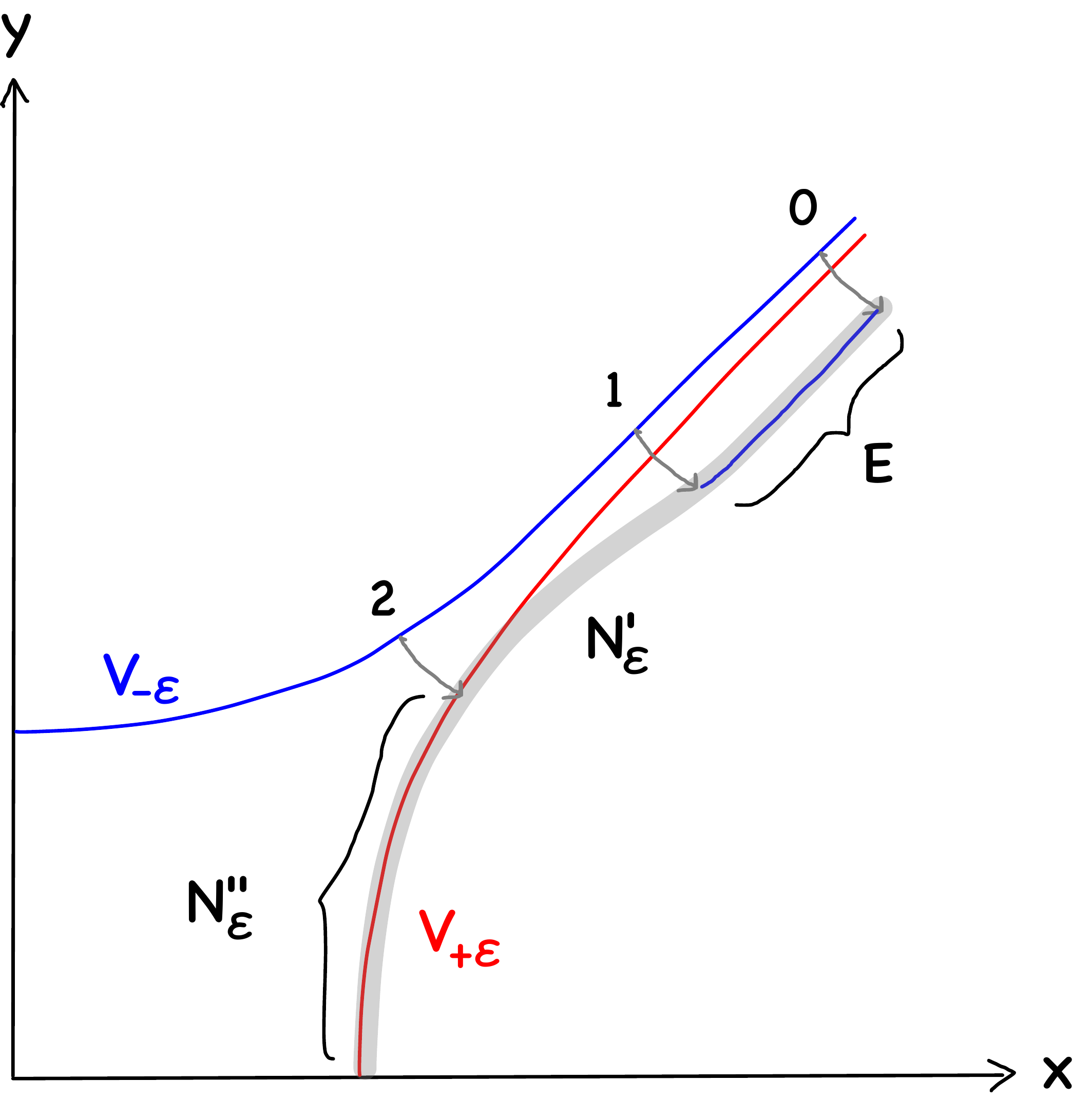}
\caption{A model for the boundary of the handle that is easy to attach. The points marked by $j=0,1,2$ illustrates the set of points with $x_1^{2}+\epsilon^{2s}x_2^{2}=\tfrac{1}{\sqrt{2}}\epsilon^{s+1}-5j\epsilon^{l}$}
\label{fig:interpol}
\end{figure}

\begin{itemize}
\item
It agrees with $V_{+\epsilon}$ in the bounded part $N''_{\epsilon}(\tilde\epsilon^{s+1})$ of $V_{+\epsilon}$  with boundary
\begin{equation}\label{e:N''}
\pa N''_\epsilon(\tilde\epsilon^{s+1})=\Omega^{T}\left(\bigl\{(x,y)\in V_{-\epsilon}\colon x_1^{2}+\epsilon^{2s}x_2^{2}= \tfrac{1}{\sqrt{2}}\epsilon^{s+1}-10\epsilon^{l}\bigr\}\right).
\end{equation}
\item
It agrees in a neighborhood of its boundary with a neighborhood in $E$ of the part of the boundary of $E$ corresponding to $x_1^{2}+\epsilon^{2s}x_2^{2}=\tfrac{1}{\sqrt{2}}\epsilon^{s+1}$.
\item
It lies in an $\Ordo(\epsilon^{2p-5s-4+l})$ $C^{2}$-neighborhood of $N_{+\epsilon}(\epsilon^{s+1})$.
\end{itemize}

The Liouville field $v$, see \ref{e:LiouvilleV}, on $\C^{n}$ then induces the contact form $2x\cdot dy+y\cdot dx$ on $N'_{\epsilon}(\tilde\epsilon^{s+1})$. Define
\[
F\colon N_{-\epsilon}(\tfrac{1}{\sqrt{2}}\epsilon^{s+1})-N_{-\epsilon}(\tfrac{1}{\sqrt{2}}\epsilon^{s+1}-5\epsilon^{l})\to
N'_\epsilon(\tilde\epsilon^{s+1})
\]
as
\begin{equation}\label{e:defF}
F(x,y)=\Omega^{T}(x,y).
\end{equation}

The Liouville vector field $v$ satisfies $L_{v}(2x\cdot dy+y\cdot dx)=2x\cdot dy+y\cdot dx$ and $(2x\cdot dy + y\cdot dx)(v)=0$. Thus
\begin{align}\label{e:Liouvillepullback}
F^{\ast}(\alpha)&=F^{\ast}(2x\cdot dy + y\cdot dx)\\\notag
&=e^{T}(2x\cdot dy + y\cdot dx)\\\notag
&=e^{T}\alpha.
\end{align}
We conclude from \eqref{e:Liouvillepullback} that $F$ takes the contact structure on $N_{-\epsilon}(\tfrac{1}{\sqrt{2}}\epsilon^{s+1})$ to the contact structure on $N'_\epsilon(\tilde\epsilon^{s+1})$, and that the contact form on the image which corresponds to $\alpha$ on $V_{-\epsilon}$ is $e^{-T}\alpha$ rather than $\alpha$ itself.

Let $N(\Lambda)$ denote the neighborhood of $\Lambda\subset Y$ which is identified with $N_{-\epsilon}(\tfrac{1}{\sqrt{2}}\epsilon^{s+1})$ and let $\hat N_{\epsilon}(\Lambda)\subset N(\Lambda)$ be the subset which corresponds to $N_{-\epsilon}(\tfrac{1}{\sqrt{2}}\epsilon^{s+1}-5\epsilon^{l})$ under this identification. Let $F$ denote the map in \eqref{e:defF}. Define the contact manifold $Y_\Lambda(\epsilon)$ as
\begin{equation}\label{e:dfnY_Lambda}
Y_\Lambda(\epsilon)=(Y-\hat N_{\epsilon}(\Lambda))\cup_F N'_\epsilon(\tilde\epsilon^{s+1}).
\end{equation}
The contact form on $Y_{\Lambda}(\epsilon)$ is given by $e^{T}\alpha$ on $Y-\hat N_{\epsilon}(\Lambda)$ and is agreeing with the contact form induced by the Liouville vector field on $N'_{\epsilon}(\tilde\epsilon^{s+1})$. In particular, this contact form agrees with $\alpha=2x\cdot dy + y\cdot dx$ in regions where $N'_{\epsilon}(\tilde\epsilon^{s+1})$ agrees with $V_{+\epsilon}$.

\begin{rmk}
The Legendrian co-core sphere $\Gamma$ in $Y_\Lambda(\epsilon)$ is the sphere
\begin{equation}\label{e:dfnGamma}
\Gamma= S_x(\epsilon^{p})\subset N'_{\epsilon}(\tilde\epsilon^{s+1}).
\end{equation}
\end{rmk}

\subsection{The construction of $X_\Lambda(\epsilon)$}\label{s:constrXLambda}
We define the symplectic cobordism $X_{\Lambda}(\epsilon)$ interpolating between $Y$ and $Y_{\Lambda}(\epsilon)$. We use notation as in Section \ref{s:constrYLambda}. Consider the bounded domain $H_\epsilon\subset\C^{n}$ bounded by the following three hyper-surfaces
\[
N_{-\epsilon}(\tfrac{1}{\sqrt{2}}\epsilon^{s+1})\cup
N'_{\epsilon}(\tilde\epsilon^{s+1})\cup
\left(\cup_{0\le t\le T}\,\,\Omega^{t}\left(\pa N_{-\epsilon}(\tfrac{1}{\sqrt{2}}\epsilon^{s+1})\right)\right),
\]
equipped with the standard symplectic form on $\C^{n}$. By \eqref{e:N'}, in a neighborhood of
\[
\cup_{0\le t\le T}\,\,\Omega^{t}\left(\pa N_{-\epsilon}(\tfrac{1}{\sqrt{2}}\epsilon^{s+1})\right),
\]
$H_\epsilon$ is (canonically) symplectomorphic to $U\times[0,T]\subset Y\times\R$, where $Y\times\R$ is the symplectization of $Y$ and where $U\subset Y$ is a small neighborhood of $\pa \hat N_{\epsilon}(\Lambda)$. We can thus attach the part $(Y-\hat N_{\epsilon}(\Lambda))\times[0,T]$ of the symplectization of $Y-\hat N_{\epsilon}(\Lambda)$ to $H_\epsilon$ in a canonical way. This gives a compact symplectic cobordism $X_{\Lambda}^{\circ}(\epsilon)$ connecting $Y_{\Lambda}(\epsilon)$ to $Y$ with symplectic form equal to $\omega_{\rm st}$ in $H_{\epsilon}\subset X_{\Lambda}^{\circ}(\epsilon)$ and equal to $e^{t}\alpha$ in the symplectization of $Y-\hat N_{\epsilon}(\Lambda)$.

In order to create a symplectic cobordism with cylindrical ends we attach the half symplectizations $Y\times[t_0,-\infty)$ and $Y_\Lambda(\epsilon)\times[0,\infty)$ to $X_{\Lambda}^{\circ}(\epsilon)$ along its concave- and convex boundary, respectively. More precisely, we use the Liouville flow to attach the symplectizations. Note that the symplectic form on $X_{\Lambda}^{\circ}(\epsilon)$ then extends in a canonical way. We call this non-compact symplectic cobordism $X_{\Lambda}(\epsilon)$ and denote its symplectic form $\omega$.

\section{Reeb chords, Reeb orbits, and a deformation of the basic handle}\label{S:Reebinbasic}
In this section we describe the Reeb chords of the co-core spheres $\Gamma$ and Reeb orbits in the contact manifold $Y_{\Lambda}(\epsilon)$.  We also define slightly deformed versions of the symplectic cobordism $X_{\Lambda}(\epsilon)$ and its contact boundary $Y_{\Lambda}(\epsilon)$ that are well suited for our study of holomorphic curves in later sections. 

Let $\Lambda\subset Y$ be a Legendrian link and consider $Y_\Lambda(\epsilon)$, constructed by applying the construction in Section \ref{s:constrYLambda} to all components of the link simultaneously and use notation as there. Note that on the piece $Y-\hat N_{\epsilon}(\Lambda)$ common to the two contact manifolds the contact forms agree (after scaling by $e^{T}=1+\Ordo(\epsilon^{2p-2s-2})$). The contact form induces an action functional on the space of curves in a contact manifold. If $\gamma$ is a curve in a contact manifold $Y$ with contact form $\alpha$, then the action of $\gamma$ is
\[
\act(\gamma)=\int_\gamma\alpha.
\]

Below we present three lemmas that describe in turn the new Reeb orbits in $Y_{\Lambda}(\epsilon)$, the Reeb chords of the co-core spheres $\Gamma\subset Y_{\Lambda}(\epsilon)$, and the new Reeb chords of a Legendrian submanifold $\Lambda_{0}\subset Y$ disjoint from $\Lambda$ considered as submanifold of $Y_{\Lambda}(\epsilon)$.

\subsection{Reeb orbits in $Y_{\Lambda}(\epsilon)$}
We consider first the most involved case of Reeb orbits. In order to describe them we use the following terminology. A {\em cyclic word} of Reeb chords of $\Lambda$ is an equivalence class of words $c_1\dots c_m$ of Reeb chords where two words are considered equivalent if one can be obtained from the other by cyclic permutation. A \emph{composable} cyclic word is a cyclic word which has a representative $c_{1}\dots c_{m}$ such that the end point of $c_j$ lies in the same component of $\Lambda$ as the start point of $c_{j+1}$ for $1\le j\le m$, where we use the convention $c_{m+1}=c_1$. If $c_1\dots c_m$ is a word of Reeb chords we write $(c_1\dots c_m)^{\circ}$ for the cyclic word represented by the word $c_1\dots c_m$. We define the action of a word of Reeb chords as
\[
\act(c_1\dots c_m)=\sum_{j=1}^{m}\act(c_j)
\]
and the action of a cyclic word as the action of any representative. 
We note that for generic contact form $\alpha$ on $Y$ the set of actions of Reeb orbits and of  words of Reeb chords is discrete. We call this set the {\em action set} of $\Lambda\subset Y$.

Let $\act_0>0$ be real number which is not in the action set of $\Lambda\subset Y$. Write $\mathcal{R}^{\circ}_{\act_0}(Y)$ for the Reeb orbits in $Y$ of action less than $\act_0$. Also write $\Omega^{\circ}\mathcal{R}_{\act_0}(\Lambda)$ for the set of composable cyclic words of Reeb chords of $\Lambda$.

\begin{lma}\label{l:neworbits}
There exists $\epsilon_0>0$ such that for $0<\epsilon<\epsilon_0$, there is a natural 1-1 correspondence 
\[ 
\mathcal{R}^{\circ}_{\act_0}(Y_{\Lambda}(\epsilon)) \ \approx \
\mathcal{R}^{\circ}_{\act_0}(Y) \ \cup \ \Omega^{\circ}{\mathcal{R}}_{\act_0}(\Lambda).
\]
In fact, any Reeb orbit $\gamma$ of $Y$ with $\act(\gamma)<\act_0$ lies outside an $\Ordo(\epsilon)$-neighborhood of $\Lambda$ with respect to the reference metric and is also a Reeb orbit of $Y_\Lambda(\epsilon)$.
\end{lma}

\begin{pf}
It is clear that any Reeb orbit in $Y$ of action $<\act_0$ stays outside an $\Ordo(\epsilon)$-neighborhood of $\Lambda$ for $\epsilon$ small enough. Also, it is straightforward to check that any Reeb orbit in $Y_\Lambda(\epsilon)$ which goes through the handle converges to a cyclic word of Reeb chords as $\epsilon\to 0$. It thus remains to show that there exists a unique orbit for each cyclic word for $\epsilon$ sufficiently small. We carry out the proof in the case when $\Lambda$ has only one connected component, the multi-component case is only notationally more difficult. 

Consider the construction of $Y_{\Lambda}(\epsilon)$, see \eqref{e:dfnY_Lambda}. Recall that $\hat N_{\epsilon}(\Lambda)\subset Y$ denotes the subset which corresponds to $N_{-\epsilon}(\tfrac{1}{\sqrt{2}}\epsilon^{s+1}-5\epsilon^{l})\subset V_{-\epsilon}$ under the identification of a neighborhood of $\Lambda\subset Y$ with a neighborhood of the core sphere $S_{y}(\epsilon^{p})\subset V_{-\epsilon}$. Let $\tilde N_{\epsilon}(\Lambda)\subset\hat N_{\epsilon}(\Lambda)$ denote the subset which corresponds to $N_{-\epsilon}(\tfrac{1}{\sqrt{2}}\epsilon^{s+1}-10\epsilon^{l})$.

By construction, the Reeb flow on $Y_{\Lambda}(\epsilon)$ can be described as follows for Reeb flow lines which come close to the co-core sphere $\Gamma$.
\begin{itemize}
\item[$(1)$] Follow the Reeb flow in $Y-\tilde N_{\epsilon}(\Lambda)$ until the flow line hits 
$$
\pa \tilde N_{\epsilon}(\Lambda)\approx\pa N_{-\epsilon}(\tfrac{1}{\sqrt{2}}\epsilon^{s+1}-10\epsilon^{l}).
$$
\item[$(2)$] Continue along the Liouville flow starting in $\pa N_{-\epsilon}(\tfrac{1}{\sqrt{2}}\epsilon^{s+1}-10\epsilon^{l})$ until it hits 
$\pa N''_\epsilon(\tilde\epsilon^{s+1})\subset V_{+\epsilon}$, see \eqref{e:N''}.
\item[$(3)$] Continue along the Reeb flow in $N''_\epsilon(\tilde\epsilon^{s+1})$ until $\pa N_\epsilon''(\tilde\epsilon^{s+1})$ is hit.
\item[$(4)$] Continue along the backwards Liouville flow starting in
$\pa N_\epsilon''(\tilde\epsilon^{s+1})$ until 
$$
\pa N_{-\epsilon}(\tfrac{1}{\sqrt{2}}\epsilon^{s+1}-10\epsilon^{l})\approx\pa\tilde N_{\epsilon}(\Lambda)
$$ 
is hit.
\item[$(5)$] Repeat $(1)-(4)$.
\end{itemize}

Consider a cyclic word of Reeb chords $(c_1\dots c_m)^{\circ}$. Write $c_j^{+}$ and $c_j^{-}$ for the endpoint of the Reeb chord $c_j$ where the Reeb vector field is outward respectively inward pointing. The boundary $\pa \tilde N_{\epsilon}(\Lambda)$ has the form $\Lambda\times S$ where $S\approx S_x(r)$, $r=\epsilon^{s+1}+\Ordo(\epsilon^{l})$ is the {\em fiber $(n-1)$-sphere}, see \eqref{eq:corenotation}. The Reeb flow on $Y$ is inward pointing on a half sphere $H_{-}\subset S$ and outward pointing on the other half sphere $H^{+}\subset S$. We write $(q,p)$ for points in $\Lambda\times S$.

Consider sets $D_{j}^{+}\times H^{-}$ and $D_{j}^{-}\times H^{+}$, where $D_j^{+}\subset \Lambda$ and $D_j^{-}\subset\Lambda$ are fixed disk neighborhoods of $c_j^{+}$ and of $c_{j}^{-}$, respectively.
Let $(q,p)\in D_{j}^{+}\times H^{-}$. Let $\Omega^{t}$ denote the Liouville flow. Then $\Omega^{T}(q,p)\in \pa N_{\epsilon}''(\tilde\epsilon^{s+1})$, see \eqref{e:N''}. Let $\tau_{\epsilon}(q,p)$ denote the time for which the Reeb flow $\Phi_{\epsilon}^{t}$ in $N_{\epsilon}''(\tilde\epsilon^{s+1})\subset V_{+\epsilon}$ starting at $\Omega^{T}(q,p)$ hits $\pa N_{\epsilon}''(\tilde\epsilon^{s+1})$, i.e.
\[
\Phi_{\epsilon}^{\tau_{\epsilon}(q,p)}(\Omega^{T}(q,p))\in\pa N_{\epsilon}''(\tilde\epsilon^{s+1}).
\]
Then $\Omega^{-T}(\Phi_{\epsilon}^{\tau_{\epsilon}(q,p)}(\Omega^{T}(q,p)))\in\Lambda\times H^{-}$. Define
\[
\Psi_{\epsilon}(q,p)=\Omega^{-T}(\Phi_{\epsilon}^{\tau_{\epsilon}(q,p)}(\Omega^{T}(q,p)))
\]
and let $I_1=D_{1}^{+}\times H^{-}$.

Consider the set
\[
J_1=\Psi_{\epsilon}(I_1)\cap (D_{2}^{-}\times H^{+}).
\]
Lemma \ref{l:flowthrhandle} implies that if $A$ is the set of $(p,q)\in D_1^{+}\times H^{-}$ such that $\Psi_{\epsilon}(q,p)\in D_{2}^{-}\times H^{+}$ then $A$ fibers over $D_1^{+}$. Here the fiber of $A$ over $q\in D_{1}^{+}$ consists of a neighborhood $D'$ of a point in $H^{-}$. Using the explicit formula \eqref{e:Reebflow} it is then straightforward to check that  $J_1$ fibers over $D_2^{-}$ with fiber over $q\in D_{2}^{-}$ a neighborhood of a point in $H^{+}$.

Consider the Reeb flow $\Phi^{t}$ in $Y-\tilde N_{\epsilon}(\Lambda)$ starting at $(q,p)\in D_{2}^{-}\times H^{+}$. Let $\tau(p,q)$ denote the time for which $\Phi^{t}(p,q)$ hits $D_2^{+}\times H^{-}$, i.e.
\[
\Phi^{\tau(q,p)}(q,p)\in D^{2}\times H^{-}.
\]
Let $\Psi(q,p)=\Phi^{\tau(q,p)}(q,p)$, and define
\[
I_2=\Psi(J_1)\cap D_{2}^{+}\times H^{-}.
\]
Since the linearized Reeb flow in $Y$ takes the tangent space $T_{c_2^{-}}\Lambda$ of $\Lambda$ at $c_2^-$ to a subspace transverse to the tangent space $T_{c_2^{+}}\Lambda$ of $\Lambda$ at $c_2^{+}$, it follows that $I_2$ fibers over $H^{-}$. The fiber over a point $p\in H^{-}$ is a subset of $D_2^{+}$. The size of this subset depends on the angle between the image of $T_{c_2^{-}}\Lambda$ and $T_{c_2^{+}}\Lambda$ and on how well the Reeb flow is approximated by its linearization. Since both these quantities are uniformly controlled for Reeb chords below a fixed action, we conclude that the size of the fiber is controlled by some constant times the size of the fiber in $J_1$. Repeating this process we produce fibered subsets $J_2, I_3,J_3,\dots,J_m$ with the property that $J_k=\Psi_{\epsilon}(I_k)\cap D_{j+1}^{-}\times H^{+}$ and $I_{k+1}=\Psi(J_k)\cap D_{k+1}^{+}\times H^{-}$. With sizes of fibers controlled as above in each step.

Consider the flow image $\Psi(J_m)$ inside $I_1=D_1^{+}\times H^{-}$. It is a subset which fibers over $H^{-}$ with fibers which are contained in $D_1^{+}$. Let $F\colon D_{1}^{+}\times H^{-}\to D_{1}^{+}\times H^{-}$ denote the map which corresponds to the result of going once around the cyclic word, i.e., $F$ is the composition
\[
\begin{CD}
I_{1} @>{\Psi_{\epsilon}}>> J_{1} @>{\Psi}>> I_{2} @>{\Psi_{\epsilon}}>> \cdots @>\Psi>> I_{m} @>\Psi_{\epsilon}>> J_{m} @>\Psi>> I_1.
\end{CD}
\]
Then $F$ contracts fibers, and since the flow takes closed subsets to closed subsets, any finite intersection of the form
\[
\bigl(F(I_1)\bigr)\cap \bigl(F\circ F(I_1)\bigr)\cap\dots\cap \bigl(F\circ\dots\circ F(I_1)\bigr)
\]
is non-empty. It follows that the corresponding infinite intersection is non-empty as well. In particular, $F$ has a fixed point $(q_0,p_0)$ which corresponds to a Reeb orbit in $Y_\Lambda(\epsilon)$. Moreover, since the map contracts sizes of fibers it follows that the $p$-coordinate $p_0$ of $(q_0,p_0)$ is unique.

Repeating the procedure using instead the backwards Reeb flow and letting $G\colon D^{+}_1\times H^{-}\to D^{+}_1\times H^{-}$ denote the map which corresponds to going once around we find, using an identical argument, a fixed point $(q_0',p_0')$ for $G$ the $q$-coordinate of which is unique. Since fixed points of $G$ correspond to Reeb orbits in $Y_{\Lambda}(\epsilon)$ as well, the lemma follows. \end{pf}

\subsection{Reeb chords of Legendrian submanifolds of $Y_{\Lambda}(\epsilon)$}
We next consider the case of Reeb chords on the co-core link $\Gamma\subset Y_{\Lambda}(\epsilon)$. In order to describe the result we introduce the following notation. Assume that the components of the attaching link $\Lambda$ are $\Lambda_{1},\dots,\Lambda_{k}$. Then the co-core link $\Gamma$ has corresponding components $\Gamma_{1},\dots,\Gamma_k$, where $\Lambda_{j}$ and $\Gamma_{j}$ lies in the same handle. Let $\mathcal{R}_{\act_0}^{ij}(\Gamma)$ denote the Reeb chords of $\Gamma$ that start on $\Gamma_i$ and ends on $\Gamma_j$ and of action at most $\act_{0}$. Similarly, let  $\Omega\mathcal{R}_{\act_0}^{ij}(\Lambda)$ denote the set of composable words $c_1\dots c_m$ of Reeb chords of $\Lambda$ where the start point of $c_1$ lies in $\Lambda_i$ and the end point of $c_m$ in $\Lambda^{j}$, and where the composable means that the end point of $c_s$ lies in the same component of $\Lambda$ as the start point of $c_{s+1}$, for $s=1,\dots,m-1$.

\begin{lma}\label{l:newchords}
Let $\act_0>0$ be real number which is not in the action set of $\Lambda\subset Y$. Then there exists $\epsilon_0>0$ such that for $0<\epsilon<\epsilon_0$, there is a natural 1-1 correspondence 
\[ 
\mathcal{R}_{\act_0}^{ij}(\Gamma) \ \approx \
\Omega\mathcal{R}_{\act_0}^{ij}(\Lambda).
\]
\end{lma}

\begin{pf}
The proof is similar to the proof of Lemma \ref{l:neworbits} but simpler (we need not use iteration). It is easy to see that any Reeb chord of $\Gamma$ converges to a word in the chords of $\Lambda$. In order to show existence and uniqueness we fix neighborhoods $D_j^{\pm}\times H^{\mp}$ of the Reeb chord end points. Furthermore we identify the boundary $\Lambda\times S$ of the neighborhood of $\Lambda\subset Y$ with the boundary of a neighborhood of $\Gamma\subset Y_{\Lambda}(\epsilon)$ using the Liouville flow without further mentioning.

Consider a Reeb flow line starting on $\Gamma$ and hitting $D_1^{-}\times H^{+}$. Note that the intersection of the flow image of $\Gamma$ with $D_1^{-}\times H^{+}$ is an $(n-1)$ dimensional subset which fibers over $D_1^{-}$. The image of this subset under the Reeb flow in $Y$ is a subset of $D_1^{+}\times H^{-}$ which fibers over $H^{-}$, since the linearized Reeb flow is transverse at endpoints. In particular, there is some point in the image which corresponds to the Reeb flow line the continuation of which would hit $\Lambda$. We then find a flow line in the handle which hits $D_2^{-}\times H^{+}$ and the image fibers over $D_2^{-}$ etc. When we reach the last Reeb chord of the word the flow image is an $(n-1)$-dimensional subset of $D_m^{+}\times H^{-}$ which fibers over $H^{-}$ and with some point corresponding to a flow line which would hit $\Lambda$. A flow line inside the handle starting at the point of the flow line which would hit $\Lambda$ is directed toward the center of the core-disk and will thus hit $\Gamma$. The lemma follows.
\end{pf}

\begin{proof}[Proof of Theorem \ref{t:Reebdyn}]
Theorem \ref{t:Reebdyn} is a restatement of Lemmas \ref{l:neworbits} and \ref{l:newchords}.	
\end{proof}

We next consider the case of a Legendrian submanifold $\Lambda_{0}\subset Y-\Lambda$ and note that this gives rise to a Legendrian submanifold $\Lambda_{0}^{+}\subset Y_{\Lambda}(\epsilon)$ provided $\epsilon>0$ is small enough. In order to describe Reeb chords of $\Lambda_{0}^{+}$ we introduce the notation $\Omega\mathcal{R}_{\act_0}(\Lambda_0;\Lambda;\Lambda_{0})$ for the set of words of composable Reeb chords $c_1\dots c_m$ such that the start point of $c_1$ and the endpoint of $c_m$ lies on $\Lambda_0$, such that other endpoints lie in $\Lambda$ with the end point of $c_s$ in the same component as the start point of $c_{s+1}$. 

\begin{lma}\label{l:newchordsL_0}
Let $\act_0>0$ be real number which is not in the action set of $\Lambda\subset Y$. Then there exists $\epsilon_0>0$ such that for $\epsilon<\epsilon_0$, there is a natural 1-1 correspondence
\[  
\mathcal{R}_{\act_0}(\Lambda_{0}^{+}) \ \approx \ \Omega\mathcal{R}_{\act_0}(\Lambda_0;\Lambda;\Lambda)
\]
In fact any Reeb chord $c$ of $\Lambda_0$ with $\act(c)<\act_0$ lies outside an $\Ordo(\epsilon)$-neighborhood of $\Lambda$ with respect to the reference metric and is also a Reeb chord of $\Lambda_{0}^{+}$ (and corresponds to the composable word $c$).
\end{lma}

\begin{pf}
Analogous to Lemma \ref{l:newchords}.
\end{pf}

Finally we consider a mixture of the previous two results. Write $\mathcal{C}_{\act_0}(\Lambda_{0}^{+},\Gamma)$ and $\mathcal{C}_{\act_0}(\Gamma,\Lambda_{0}^{+})$ for the sets of Reeb chords of action at most $\act_0$ connecting $\Lambda_{0}^{+}$ to $\Gamma$ and vice versa. Write $\Omega\mathcal{R}_{\act_0}(\Lambda_0;\Lambda)$ for the set of words of composable Reeb chords $c_1\dots c_m$ of action at most $\act_0$ such that the start point of $c_1$ lies on $\Lambda_0$, such that other endpoints lie in $\Lambda$ with the end point of $c_s$ in the same component as the start point of $c_{s+1}$. Similarly, write $\Omega\mathcal{R}_{\act_0}(\Lambda;\Lambda_0)$ for the set of words of composable Reeb chords $c_1\dots c_m$ such that the end point of $c_m$ lies on $\Lambda_0$, such that other endpoints lie in $\Lambda$ with the end point of $c_s$ in the same component as the start point of $c_{s+1}$.

\begin{lma}\label{l:newchordsmix}
Let $\act_0>0$ be real number which is not in the action set of $\Lambda\subset Y$. Then there exists $\epsilon_0>0$ such that for $\epsilon<\epsilon_0$, there are a natural 1-1 correspondences
\begin{align*}
\mathcal{C}_{\act_0}(\Lambda_{0}^{+},\Gamma)& \ \approx \
\Omega\mathcal{R}_{\act_0}(\Lambda_0;\Lambda),\\
\mathcal{C}_{\act_0}(\Gamma,\Lambda_{0}^{+})& \ \approx \
\Omega\mathcal{R}_{\act_0}(\Lambda;\Lambda_0)
\end{align*}
\end{lma}

\begin{pf}
Analogous to Lemma \ref{l:newchords}.
\end{pf}

\begin{rmk}\label{r:paramindep}
Note that none of the Lemmas \ref{l:neworbits}, \ref{l:newchords}, \ref{l:newchordsL_0}, or \ref{l:newchordsmix} use any specific property of the identification map of the Legendrian sphere $\Lambda$ with the core sphere $S_y(\epsilon^{p})\subset V_{-\epsilon}$. In particular these results hold for any such identification map. We will use this fact below.
\end{rmk}

\section{Auxiliary cobordisms and basic holomorphic strips}\label{s:buildcob'}
In this section we construct auxiliary cobordisms that are adapted to Reeb chords, we equip them with a specific almost complex structure for which it is trivial to find certain basic holomorphic strips that are the starting point for our curve counting in Section \ref{A:D}.

More precisely, if $c$ is a Reeb chord of the attaching sphere $\Lambda$ and if $c'$ is the corresponding Reeb chord of the co-core sphere $\Gamma$, see Lemma \ref{l:newchords}, then we construct auxiliary cobordisms $X_{\Lambda}(\epsilon;c)$ connecting $Y$ to $Y_{\Lambda}(\epsilon;c)$, where $X_{\Lambda}(\epsilon;c)$ is a small deformation of the cobordism $X_{\Lambda}(\epsilon)$ constructed in Section \ref{s:constrXLambda} and $Y_{\Lambda}(\epsilon;c)$ is a small deformation of $Y_{\Lambda}(\epsilon)$ as constructed in Section \ref{s:constrYLambda}. Then we equip $X_{\Lambda}(\epsilon;c)$ with an almost complex structure in which there is an obvious holomorphic strip with boundary on $L\cup C$ that interpolates between $c'$ and $c$. Here $L$ is the core disk and $C$ is the co-core disk.

Also, in complete analogy with this construction just mentioned,  if $\Lambda_0\subset Y-\Lambda$ is a Legendrian submanifold, if $a$ is a Reeb chord connecting $\Lambda_0$ to $\Lambda$ or vice versa, and if $a'$ is the corresponding Reeb chord connecting $\Lambda_{0}^{+}$ and $\Gamma$, see Lemma \ref{l:newchordsmix}, then we define a cobordism $X_{\Lambda}(\epsilon,a)$ adapted to $a$ with an almost complex structure for which there is a holomorphic strip with boundary on $L\cup C\cup(\R\times\Lambda_0)$ interpolating between $a'$ and $a$.

\subsection{Deformed hypersurfaces}
We consider a perturbation of the defining equation for $V_{\pm\epsilon}$. Let $p$ and $s$ be integers with $5s+5<p$, as in Section \ref{s:constrYLambda}. Fix an integer $q$ with $p-s<q<p$ and let $\beta_{\epsilon}\colon \R_{+}\to[0,1]$ be a cut off function with the following properties:
\begin{align*}
\beta_{\epsilon}(r)&=0 \text{ for } 0\le r\le\epsilon^{q},\\
\beta_{\epsilon}(r)&=1 \text{ for } r\ge 2\epsilon^{q}, \text{ and}\\
|d^{(k)}\beta_{\epsilon}|&=\Ordo(\epsilon^{-kq}),
\end{align*}
where $d^{(k)}\beta_{\epsilon}$ denotes the $k^{\rm th}$ derivative of $\beta_{\epsilon}$.

Write $r_2=\sqrt{x_2^{2}+y^{2}_2}$ and define the hyper-surfaces $\hat V_{\pm\epsilon}\subset\C^n$ as
\begin{equation}\label{e:dfnV'}
\hat V_{\pm\epsilon}=\bigl\{(x,y)\in\C^n\colon 2 x_1^2-y_1^2+
\beta_{\epsilon}(r_2)\epsilon^{2s}\left(
2x_2^{2}-y_2^{2}\right)=\pm\epsilon^{2p}\bigr\}.
\end{equation}
Let
\[
I_{\pm\epsilon}\subset\hat V_{\pm\epsilon}
\]
denote the subset where $\beta_{\epsilon}(r_2)<1$.

We think of $\hat V_{\pm\epsilon}$ as a perturbation of $V_{\pm\epsilon}$. Note that the perturbed and the non-perturbed hypersurfaces agree for $(x_2,y_2)=(0,0)$ and for $r_2\ge \epsilon^{q}$.

Consider the Liouville vector field $v=2x\cdot \pa_x-y\cdot\pa_y$, see \eqref{e:LiouvilleV}, and the normal vector field $n$ of $\hat V_{\pm\epsilon}$ given by
\begin{align*}
n=&2x_1\pa_{x_1}-y_1\pa_{y_1} +\epsilon^{2s}\beta(r_2)(2x_2\cdot\pa_{x_2}-y_2\cdot\pa_{y_2})\\ &+\epsilon^{2s}(2x_2^{2}-y_2^{2})\beta'(r_2)\frac{1}{r_2}(x_2\cdot\pa_{x_2}+y_2\cdot\pa_{y_2}).
\end{align*}
Then
\begin{align*}
n\cdot v=
2x_1^{2}+y_{1}^{2}+\epsilon^{2s}\beta(r_2)(2x_2^{2}+y_{2}^2)
+\epsilon^{2s}(2x_2^{2}-y_2^{2})^{2}\beta'(r_2)\frac{1}{r_2}>0
\end{align*}
Hence $v$ is transverse to $\hat V_{\pm\epsilon}$ and induces the contact form $\alpha=2x\cdot dy+y\cdot dx$ on these hypersurfaces.

\subsection{Coordinates on Reeb chords and coordinates}\label{ss:attachingmap}
Let $c$ be a Reeb chord of action $\act(c)<\act_0$. Let $c^{-}$ denote the end point of $c$ where the Reeb vector field points into $c$ and let $c^{+}$ denote the other end point of $c$. We consider the cases when at least one of the endpoints lie in $\Lambda$ and the other either also lies in $\Lambda$ or in some other Legendrian submanifold $\Lambda_{0}$.

Fix a tubular neighborhood $U(c)\subset Y$ which consist of all points of distance $\ll\epsilon_{0}$ from $c$. After possibly shrinking $U(c)$, we find coordinates $(\xi^{-},\eta^{-},\zeta^{-})\in\C^{n-1}\times\R$, i.e., $\xi^{-}+i\eta^{-}\in\C^{n-1}$ and $\zeta^{-}\in\R$, on $U(c)$ with the following properties.
\begin{itemize}
\item The Reeb chord $c$ corresponds to
\[
\{(\xi^{-},\eta^{-},\zeta^{-})\colon \xi^{-}=\eta^{-}=0,\,\,0\le\zeta^{-}\le T\}.
\]
\item The neighborhood $U(c)$ corresponds to
\[
\{(\xi^{-},\eta^{-},\zeta^{-})\colon(\xi^{-})^{2}+(\eta^{-})^{2}\le r^{2},\; -r<\zeta^{-}<T+r\}.
\]
\item The contact form $\lambda$ on $Y$ is given by
\[
\lambda = d\zeta^{-} -\eta^{-}\cdot d\xi^{-}.
\]
\item The Legendrian submanifold $\Lambda$ at the start point $c^{-}$ of $c$ is given by
\[
\{(\xi^{-},\eta^{-},\zeta^{-})\colon\eta^{-}=0,\zeta^{-}=0\}.
\]
\end{itemize}
The existence of such coordinates follows from a standard application of Moser's lemma.

In the case that the endpoint $c^{+}$ lies in $\Lambda$ then we next change $\Lambda$ by a small Legendrian isotopy supported in an small neighborhood of $c^{+}$ in the following way, see Figure \ref{fig:rotate}. 
\begin{figure}[ht]
\centering
\includegraphics[width=.40\linewidth]{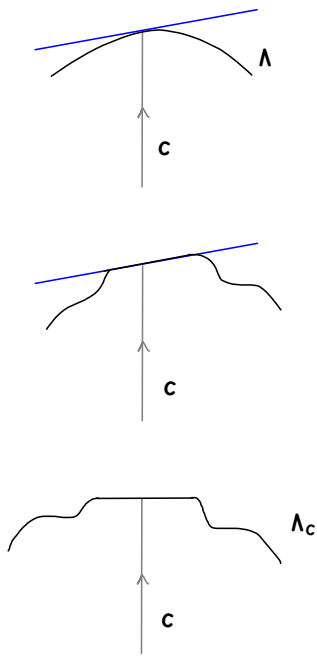}
\caption{The three pictures illustrates the start point, an intermediate instance, and the end point of the isotopy from $\Lambda$ to $\Lambda_{c}$.}
\label{fig:rotate}
\end{figure}

The Reeb flow is generic with respect to $\Lambda$, and therefore the tangent space $T_{c^{+}}\Lambda$ corresponds to a Lagrangian subspace of $\C^{n-1}\approx\{(\xi^{-},\eta^{-},\zeta^{-})\colon\zeta^{-}=T\}$ which is transverse to the subspace
\[
\{(\xi^{-},\eta^{-},\zeta^{-})\colon\zeta^{-}=T,\,\,\eta^{-}=0\}.
\]
Change $\Lambda$ by a Legendrian isotopy of small support which first makes it agree with its tangent space at $c^{+}$ and which then rotates this tangent space in $\C^{n-1}$ (the shortest way) to the subspace
\[
\{(\xi^{-},\eta^{-},\zeta^{-})\colon \zeta^{-}=0,\,\,\xi^{-}=0\}.
\]
We write $\Lambda_{c}$ for the Legendrian submanifold which results from this definition. Then $c$ is still a Reeb chord of $\Lambda_{c}$ which agrees with
\[
\{(\xi^{-},\eta^{-},\zeta^{-})\colon \xi^{-}=0,\zeta^{-}=T\}
\]
near $c^{+}$. Furthermore, if the support of the Legendrian isotopy is sufficiently small, then Reeb chords of $\Lambda$ and $\Lambda_{c}$ of action $<\act_0$ agree.

Let $(\xi^{+},\eta^{+},\zeta^{+})\in\C^{n-1}\times\R$ be coordinates similar to those above but based at the end point $c^{+}$. In these coordinates $\Lambda_{c}$ at $c^{+}$ (in the case that $c^{-}\in\Lambda_{0}$, we take $\Lambda_{c}=\Lambda$ without deformation) corresponds to
\[
\{(\xi^{+},\eta^{+},\zeta^{+})\colon\zeta^{+}=T,\,\,\eta^{+}=0\},
\]
if $c^{-}\in \Lambda_{c}$ the at $c^{-}$ $\Lambda_{c}$ corresponds to
\[
\{(\xi^{+},\eta^{+},\zeta^{+})\colon\zeta^{+}=0,\,\,\xi^{+}=0\},
\]
and the contact form on $Y$ is expressed as
\[
\lambda=d\zeta^{+}-\xi^{+}\cdot d\eta^{+}.
\]

In case both endpoints lie in $\Lambda$ then the coordinates on $U(c)$ based at $c^{+}$ and at $c^{-}$ respectively are related by the coordinate change
\begin{equation}\label{e:coordch-to+}
\begin{cases}
\zeta^{+}&=\zeta^{-}-\eta^{-}\cdot\xi^{-}\\
\xi^{+}&=\eta^{-}\\
\eta^{+}&=-\xi^{-}
\end{cases}
\end{equation}
with inverse
\begin{equation}\label{e:coordch+to-}
\begin{cases}
\zeta^{-}&=\zeta^{+}-\eta^{+}\cdot\xi^{+}\\
\xi^{-}&=-\eta^{+}\\
\eta^{-}&=\xi^{+}
\end{cases}.
\end{equation}

\subsection{Identification maps, cobordisms, and Reeb dynamics}\label{sec:identifications}
In Section \ref{s:constrXLambda} we constructed the cobordism $X_{\Lambda}(\epsilon)$ and the new contact manifold $Y_{\Lambda}(\epsilon)$ by identifying $\Lambda\subset Y$ with the core sphere $S_{y}(\epsilon^{p})\subset V_{-\epsilon}$. In this section we will use that construction with specific identification maps which are adapted to a fixed Reeb chord $c$ of $\Lambda$. 
We denote the cobordism and new contact manifold corresponding to the attaching map adapted to the Reeb chord $c$ $X_{\Lambda}(\epsilon;c)$ and $Y_{\Lambda}(\epsilon;c)$, respectively. Further we will construct cobordisms $\hat X_{\Lambda}(\epsilon;c)$ which are small deformations of $X_{\Lambda}(\epsilon;c)$ by identifying $\Lambda_{c}$ with the core sphere $\hat S_{y}(\epsilon^{p})\subset \hat V_{-\epsilon}$. (We use the notation $\hat S_{y}(\epsilon^{p})=\hat V_{-\epsilon}\cap i\R^{n}$.) There are parallel but simper constructions adapted to Reeb chords $a$ between $\Lambda$ and $\Lambda_{0}$.

The specifics of the identifications of $\Lambda$ and $\Lambda_{c}$ with the corresponding core spheres $S_{y}(\epsilon^{p})$ and $\hat S_{y}(\epsilon^{p})$, respectively, are as follows. Note that both $S_{y}(\epsilon^{p})\subset V_{-\epsilon}\subset\C^{n}$ and $\hat S_{y}(\epsilon^{p})\subset \hat V_{-\epsilon}\subset\C^{n}$ are subsets of $i\R^{n}$. Let
\[
y=(y_1,y_2)\in\R\times\R^{n-1}
\]
be coordinates on $i\R^{n}$ and write $y_{2}=(y_{2;1},\dots,y_{2;n-1})$. Then, by definition of $\hat V_{-\epsilon}$,
\[
S_{y}(\epsilon^{p})\cap \{y\colon |y_2|>2\epsilon^{q}\}=
\hat S_{y}(\epsilon^{p})\cap \{y\colon |y_2|>2\epsilon^{q}\}.
\]
We use identification maps $\phi_{c}\colon\Lambda\to S_{y}(\epsilon^{p})$ and $\hat \phi_{c}\colon\Lambda\to \hat S_{y}(\epsilon^{p})$ with the following properties, see Figure \ref{fig:chordendpoints}:
\begin{itemize}
\item $\phi_{c}(c^{\pm})=\hat\phi_{c}(c^{\pm})=(\mp\epsilon^{p},0)$.
\item If $b$ is a Reeb chord of $\Lambda$ with $\act(b)\ne \act(c)$ then
\[
\phi_{c}(b^{\pm})=\hat\phi_{c}(b^{\pm})=(y_1(b^{\pm}),y_2(b^{\pm})),
\]
where $y_1(b^{\pm})\ne 0$ and $|y_{2;1}(b^{\pm})|>\epsilon^{p-s+1}$. 
\end{itemize}

\begin{figure}[ht]
\centering
\includegraphics[width=.55\linewidth]{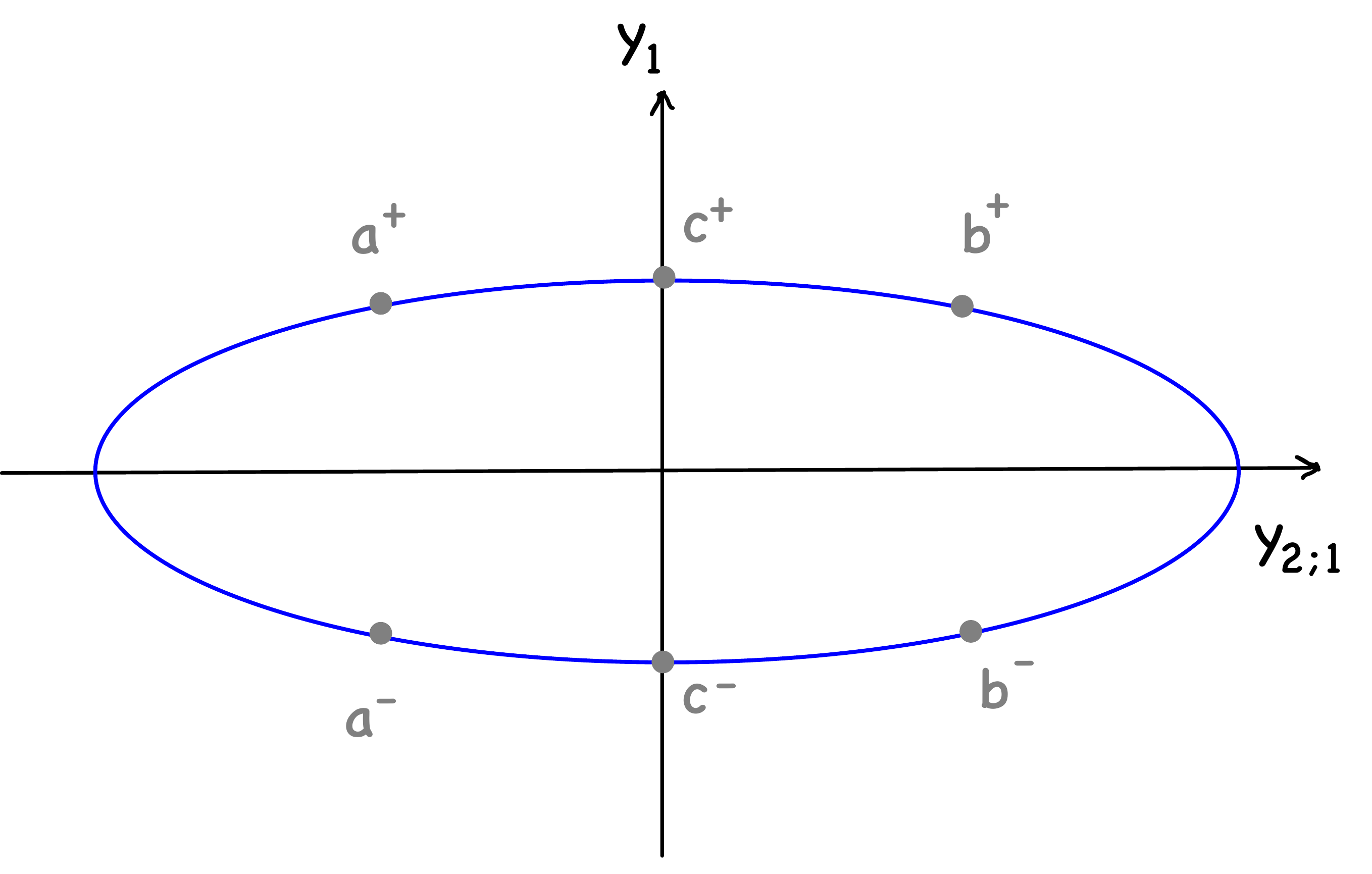}
\caption{We parameterize the Legendrian sphere $\Lambda$ in such a way that Reeb chord endpoints lie at special points in the model of the handle. The endpoints of the distinguished Reeb chord lie in the middle in the flat region, other Reeb chord endpoints lie to the left or right of it in the $y_{2;1}$-direction, as illustrated for the chords $a$ and $b$ in the picture.}
\label{fig:chordendpoints}
\end{figure}

We repeat the constructions of $Y_\Lambda(\epsilon)$ and of $X_{\Lambda}(\epsilon)$ in Sections \ref{s:constrYLambda} and \ref{s:constrXLambda}, respectively, on the one hand using hypersurfaces $V_{\pm\epsilon}$ as there and on the other using $\hat V_{\pm\epsilon}$ instead of $V_{\pm\epsilon}$. More precisely, we write $X_{\Lambda}(\epsilon;c)$ and $\hat X_{\Lambda}(\epsilon;c)$ for the cobordisms constructed using the identification maps $\phi_{c}\colon\Lambda\to S_{y}(\epsilon^{p})$ and $\hat\phi_{c}\colon\Lambda_{c}\to \hat S_{y}(\epsilon^{p})$, respectively, and we write $Y_{\Lambda}(\epsilon;c)$ and $\hat Y_{\Lambda}(\epsilon;c)$ for the new contact manifolds which are convex boundaries of $X_{\Lambda}(\epsilon;c)$ and $\hat X_{\Lambda}(\epsilon;c)$, respectively. We also write $\Gamma$ and $\hat{\Gamma}$ for the co-core spheres in $Y_{\Lambda}(\epsilon)$ and $\hat Y_{\Lambda}(\epsilon)$, respectively. 

As was pointed out in Remark \ref{r:paramindep} the Reeb chords of $\Gamma$ and the Reeb orbits in $Y_{\Lambda}(\epsilon)$ are independent of the parametrization of the attaching sphere. It follows from this that the results on Reeb orbits and Reeb chords in Section \ref{S:Reebinbasic} hold for $\Gamma$ and $Y_{\Lambda}(\epsilon)$ as defined by the attaching maps in the present section as well.
(Note that our small deformations of $\Lambda$ do not introduce any new Reeb chords)
In fact, it is not hard to see that the corresponding results hold for $\hat{\Gamma}$ and $\hat Y_{\Lambda}(\epsilon)$ as well. More precisely, we have the following:

\begin{lma}
Lemmas  \ref{l:neworbits}, \ref{l:newchords}, \ref{l:newchordsL_0}, and \ref{l:newchordsmix} hold with $\Gamma=\hat\Gamma$ and $Y_{\Lambda}(\epsilon)=\hat Y_{\Lambda}(\epsilon)$.
\end{lma}

\begin{pf}
In order to compare the Reeb flows in $Y_{\Lambda}(\epsilon)$ and $\hat Y_{\Lambda}(\epsilon)$ we note that the new flow is determined by the flow through the handle viewed as a map from the boundary of a tubular neighborhood of the attaching sphere to itself. If we rescale the handle bounded by $V_{\pm\epsilon}$ so that the major axes of the ellipse in the boundary of the gluing region $N_{\epsilon}''(\epsilon^{s+1})$, see Section \ref{s:constrYLambda} for notation, equals $1$ then its minor axis is of length $\epsilon^{s}$. In this setting the $\hat V_{\pm\epsilon}$ corresponds to flattening the ellipse in a neighborhood of its minor axis of width $\epsilon^{q-1}$ where $q>3s$. Thus boundary regions of the attaching handles are likewise very close. 

Using the rescaled Reeb flow as in \eqref{e:Reebflow} and the counterpart for $\hat V_{\pm\epsilon}$ we find that the flow maps through the handle (estimated crudely as the maximal flow time times the norm of maximal difference between the vector fields) for the boundary rescaled as above are $\Ordo(\epsilon^{3s-1})$ which is small compared also to the minor axis. We conclude that the flow maps can be made arbitrarily close by taking $\epsilon$ sufficiently small. The lemma follows.     
\end{pf}

\begin{rmk}\label{r:dfnc'}
The Reeb chord of $\hat\Gamma\subset Y_{\Lambda}(\epsilon;c)$ which corresponds to the Reeb chord $c$ of $\Lambda\subset Y$ will be denoted $c'$.
\end{rmk}

We end this section by noting that the Reeb flow along $\hat V_{\pm\epsilon}$ admits a very simple description in the region where $r_2<\epsilon^{q}$. Here $\hat V_{\pm\epsilon}$ is a product,
\[
\hat V_{\pm\epsilon}=\gamma_{\pm\epsilon}\times D(\epsilon^{q})\subset \C\times\C^{n-1},
\]
where $\gamma_{\pm\epsilon}$ is the curve
\[
\gamma_{\pm\epsilon}=\{(x_1,y_1)\colon 2x_1^{2}-y_1^{2}=\epsilon^{2p}\}
\]
and where $D(\epsilon^{q})$ is a disk of radius $\epsilon^{q}$ in $\C^{n-1}$. Consequently, the Reeb flow lines in this region are of the form
\[
\gamma_{\pm\epsilon}\times\{p\},
\]
where $p\in D(\epsilon^{q})$. Defining a local coordinate $z$ in this region as
\[
z=\int_{\gamma_{-\epsilon}|_{[0,t]}} 2x_1dy_1+y_1dx_1,
\]
where $\gamma_{-\epsilon}(0)$ lies in the Legendrian core sphere. Then in coordinates $(z,x_2,y_2)\in\R\times\C^{n-1}$ on $\gamma_{\pm\epsilon}\times D(\epsilon^{q})$, the contact form on $\hat V_{\pm\epsilon}$ is given by
\begin{equation}\label{e:handlecoord}
dz-2x_2\cdot dy_2 +y_2\cdot dx_2.
\end{equation}

\subsection{An almost complex structure on $\hat X_{\Lambda}(\epsilon;c)$}\label{sec:acs}
Let $c$ be a Reeb chord of $\Lambda\subset Y$ and consider the cobordism $\hat X_{\Lambda}(\epsilon;c)$, see Section \ref{s:constrXLambda}. Let $\omega$ denote the symplectic form on $\hat X_{\Lambda}(\epsilon;c)$. In this section we endow $\hat X_{\Lambda}(\epsilon;c)$ with an almost complex structure compatible with $\omega$.

Consider the neighborhood $U(c)\subset Y$ of the Reeb chord $c$ as in Section \ref{ss:attachingmap} with coordinates $(\zeta^{-},\xi^{-},\eta^{-})\in\R\times\C^{n-1}$ in which the contact form is given by
\[
\lambda=d\zeta^{-}-\eta^{-}\cdot d\xi^{-}.
\]
Consider the contact planes $\krn(\lambda)$ in $U(c)$. If $\pi\colon \R\times\C^{n-1}\to\C^{n-1}$ denotes the projection then $\pi|_{\krn(\lambda)}$ is an isomorphism. Define the complex structure $\hat J$ on the contact planes by
\begin{equation}\label{e:dfnhatJ}
\hat J = \pi^{-1}\circ J_0\circ \pi,
\end{equation}
where $J_0$ denotes the standard complex structure on $\C^{n-1}$. We note that \eqref{e:dfnhatJ} in coordinates $(\zeta^{+},\xi^{+},\eta^{+})$ on $U(c)$ instead of $(\zeta^{-},\xi^{-},\eta^{-})$ leads to the same complex structure on $\krn(\lambda)$.

Using the complex structure $\hat J$ on $\krn(\lambda)$ we will construct an almost complex structure $J_{c}$ on $X_{\Lambda}(\epsilon;c)$ for sufficiently small $\epsilon$. Consider the neighborhood map $\Phi_{c}\colon N_{\epsilon}(\Lambda)\to \hat V_{-\epsilon}$ which agrees with the map $\phi_{c}\colon \Lambda\to \hat S_{y}(\epsilon^{p})$ on $\Lambda$.

Choose $\epsilon>0$ sufficiently small so that $U(c)$ contains
\[
\{(\zeta^{\pm},\xi^{\pm},\eta^{\pm})\colon (\xi^{\pm})^{2}+(\eta^{\pm})^{2}\le \epsilon^{2q}\}.
\]
Then the complex structure $\hat J$ of \eqref{e:dfnhatJ} gives a complex structure on the contact planes in the $\epsilon^{q}$-neighborhood of the central Reeb flow line
\[
\gamma_{\pm\epsilon}\times\{0\}
\]
in $\hat V_{\pm\epsilon}$.

Using $\hat J$ we define an almost complex structure $J_c$ on the compact part $\hat X^{\circ}_{\Lambda}(\epsilon,c)$, see Section \ref{s:constrXLambda} for notation, as follows. Consider the contact $(2n-2)$-plane field $\krn(\lambda)$ in the $\epsilon^{q}$-neighborhoods of $\gamma_{\pm\epsilon}\times\{0\}$ in $V_{\pm\epsilon}$. Extend this plane field to a $(2n-2)$-plane field $\Xi$ defined in an $\epsilon^{q}$-neighborhood $W(\epsilon^{q})$ of
\[
\hat X^{\circ}_{\Lambda}(\epsilon,c)\cap \{(x_1,y_1,x_2,y_2)\colon x_2=0,\,y_2=0\}
\]
with the property that the planes in $\Xi$ are everywhere transverse to the $(x_1,y_1)$-plane. To see that such an extension exists, note that the angle $\theta$ between the contact planes $\krn(\lambda)$ and the $(x_1,y_1)$-plane satisfies $\theta=\frac{\pi}{2}+\Ordo(\epsilon^{q})$ at all points in the $\epsilon^{q}$-neighborhood of $\gamma_{\pm\epsilon}\times\{0\}$.

Let $J_c$ be the unique almost complex structure on $W(\epsilon^{q})$ which has the following properties
\[
J_{c}(\pa_{x_1})=\pa_{y_1},\quad
J_{c}(\Xi)=\Xi,\quad\text{and }J_{c}|_{\Xi}=\pi^{-1}\circ J_0\circ \pi,
\]
where $\pi\colon \Xi\to\C^{n-1}$ is the isomorphism induced by the projection which forgets the $(x_1,y_1)$-coordinates.

Consider the almost complex structure $\hat J$ on the symplectizations of $\hat V_{\pm\epsilon}$ induced by the almost complex structure $\hat J$, see \eqref{e:dfnhatJ}, in the standard way. Then $J_c$ agrees with $\hat J$ when these are restricted to contact planes in $\epsilon^{q}$-neighborhoods of $c$ in $\hat V_{-\epsilon}$ and of $c'$ in $\hat V_{+\epsilon}$, respectively. Furthermore since both complex structures $\hat J$  and $J_c$ leaves the tangent $2$-planes spanned by the Reeb vector field and the symplectization direction invariant we find that after changing $J_c$ inside these planes only, $J_c$ and $\hat J$ agree completely in a neighborhood of
\[
(W(\epsilon^{q})\cap \hat V_{-\epsilon})\cup
(W(\epsilon^{q})\cap \hat V_{+\epsilon})\cup
\bigl(\cup_{0\le t\le T} \Omega^{t}(\pa N_{\epsilon}(\Lambda))\bigr),
\]
where $\Omega^{t}$ denotes the time $t$ Liouville flow as in \eqref{e:defF}.
Note in particular that tangent $2$-planes spanned by the Reeb vector field and the symplectization direction are $J_{c}$-complex also for the deformed complex structure $J_{c}$.

We thus have an almost complex structure $J_c$ in $W(\epsilon^{q})$ which agrees with the symplectization almost complex structure of $\hat V_{\pm\epsilon}$ near $W(\epsilon^{q})\cap V_{\pm\epsilon}$ and with the almost complex structure in $\hat V_{-\epsilon}\times[0,T]$ near $(\cup_{0\le t\le T} \Omega^{t}(\pa N(\Lambda))\bigr)$.

Let
\[
W^{\circ}(\epsilon^{q})= W(\epsilon^{q})\cup (\cup_{0\le t\le T} \Omega^{t}(U(c))\bigr).
\]
Using the symplectization to extend $J_c$ in an $\R_{\pm}$-invariant way we get an extension of $J_{c}$ to
\begin{equation}\label{e:dfnA(c)}
A(c)=
(U(c)\times\R_{-})\cup W^{\circ}(\epsilon^{q})\cup(U(c')\times\R_{+})\subset \hat X_{\Lambda}(\epsilon;c).
\end{equation}
Finally, we extend $J_c$ to an almost complex structure in the rest of the cobordism in such a way that it is adjusted to $\omega$. (We shall not need to consider details of that extension.)

\begin{lma}\label{l:lochol}
Let $u\colon D\to A(c)$ be a $J$-holomorphic map with image in $W(\epsilon^{q})$, in $U(c)$, or in $U(c')$. Let $\pi$ denote the projection to $\C^{n-1}$ (in either coordinate system) then $\pi\circ u$ is $J_0$-holomorphic.
\end{lma}

\begin{pf}
Write $du=du^{1}+du^{2}$ where $du^{1}$ takes values in $(t,\zeta^{\pm})$-plane if we use $U(c)$- or $U(c')$-coordinates and in the $(x_1,y_1)$-plane if we use $W(\epsilon^{q})$-coordinates, and where $du^{2}$ takes values in the contact plane if we use $U(c)$- or $U(c')$-coordinates and in $\Xi$ if we use $W(\epsilon^{q})$-coordinates. Since the $(t,\zeta^{\pm})$-plane and the $(x_1,y_1)$-plane are orthogonal to $\C^{n-1}$ in the respective coordinates, we have
\[
\pi\circ du^{2} = d(\pi\circ  u)
\]
and consequently,
\begin{align*}
\pi\circ du + J_0\circ (\pi\circ du)\circ j&=\pi(du^{2}+\pi^{-1}J_0\pi\circ du^{2}\circ j)\\
&=\pi(du^{2}+J_{c}\circ du^{2}\circ j)=0,
\end{align*}
since $u$ is $J_c$-holomorphic.
\end{pf}

\begin{rmk}
Recall that $\hat X_{\lambda}(\epsilon)$ contains two natural Lagrangian submanifolds both diffeomorphic to $\R^{n}$:
\begin{itemize}
\item $L$ which consists of $(i\,\R^n\cap \hat X_{\Lambda}^{\circ}(\epsilon))\cup (\Lambda\times\R_-)$.
\item $C$ which consists of $(\R^n\cap  \hat X_{\Lambda}^{\circ}(\epsilon))\cup (\Gamma\times\R_+)$.
\end{itemize}
Note also that $L\cap C$ consists of one point which corresponds to the origin in $\C^{n}$.
\end{rmk}

\begin{lma}\label{l:S(c)}
There exists a $J_{c}$-holomorphic strip $S_c\subset \hat X_{\Lambda}(\epsilon;c)$ with boundary on $C\cup L$, which is asymptotic to $c'$ at its positive puncture, which has two corners at $C\cap L$, which is asymptotic to $c$ at its negative puncture, which agrees with the Reeb chord strip of $c'$ in the upper end and that of $c$ in lower ends, respectively, and which consists of two quadrants in the $(x_1,y_1)$-plane inside $W(\epsilon^{q})$.
\end{lma}

\begin{pf}
The tangent planes of $S_c$ are $J_{c}$-complex by construction.
\end{pf}

\subsection{Analogous cobordisms adapted to a mixed Reeb chord}\label{sec:mixedanalogue}
Let $a$ be a Reeb chord connecting $\Lambda_0$ to $\Lambda$ or vice versa and let $a'$ denote the corresponding Reeb chord connecting $\Lambda_0$ and $\Gamma$. Adapting the handle and the attaching map to the end point of $a$ which lies on $\Lambda$ so that it plays the role of $c^{\pm}$ in Section \ref{sec:identifications} we construct a cobordism $\hat X(\epsilon;a)$ connecting $\hat Y(\epsilon;a)$ to $Y$. Furthermore, we find, as in Section \ref{sec:acs}, a complex structure $J_{a}$ such that the following holds.

\begin{lma}\label{l:S(a)}
There exists a $J_{a}$-holomorphic strip $S_a\subset \hat X_{\Lambda}(\epsilon; a)$ with boundary on $(\Lambda_0\times\R)\cup C\cup L$, which is asymptotic to $a'$ at its positive puncture, which has one corner at $C\cap L$, which is asymptotic to $a$ at its negative puncture, which agrees with the Reeb chord strip of $a'$ in the upper end and that of $a$ in lower ends, respectively, and which consists of two quadrants in the $(x_1,y_1)$-plane inside $W(\epsilon^{q})$.
\end{lma}

\begin{pf}
The tangent planes of $S_a$ are $J_{a}$-complex by construction.
\end{pf}

\section{Counting holomorphic disks}\label{A:D}
In this section we count holomorphic disks in the symplectic cobordism $X_{\Lambda}(\epsilon)$. We describe the exact types of disks below. Our main result shows that the count of these disks equals $\pm 1$. Our argument is as follows. First we show that the basic disks constructed in Sections \ref{sec:acs} and \ref{sec:mixedanalogue} are unique and transversely cut out. Then we glue all other curves inductively from these pieces. In doing so we must both deform the cobordism and look at splittings at the boundary of 1-dimensional moduli spaces. We are able to control these phenomena because of the small action of the disks that we consider. We will carry out the proof for Reeb chords connecting $\Lambda$ to itself and associated chords connecting $\Gamma$ to itself and orbits, and only state the results for the analogous but somewhat simpler case of chords connecting $\Lambda_{0}$ to $\Lambda$ and associated chords connecting $\Lambda_{0}^{+}$ to itself and $\Lambda^{+}$ to $\Gamma$.  

\subsection{Notation for disks}
Recall the following notation: $\Lambda\subset Y$ is the Legendrian attaching sphere, $\Gamma\subset Y_\Lambda(\epsilon)$ is the Legendrian co-core sphere, $L\subset X_{\Lambda}(\epsilon)$ is the Lagrangian $n$-plane consisting of $\Lambda\times\R_{-}$ and the core disk, and $C\subset X_{\Lambda}(\epsilon)$ for the Lagrangian $n$-plane which consists of $\Gamma\times\R+$ and the co-core disk. When studying disk with two positive punctures on $\Gamma$ below we will use three slightly different copies of $C$ (parallel copies, as in the definition of product in wrapped Floer cohomology, see \cite{EL}). More precisely, let $C=C_{0}$ and let $C_{1}$ be an $\epsilon$-shift of $C$ along the differential of a small function with one minimum near the intersection point $C\cap L$, and let $C_{2}$ be a similar $\epsilon^{2}$-shift of $C_{1}$. Then $\Gamma_{2}$ and $\Gamma_{1}$ are a small push offs of $\Gamma$ in the Reeb direction and the set of Reeb chords $\Gamma_{1}\to \Gamma_{0}$ and $\Gamma_{2}\to \Gamma_{1}$ are both canonically identified with the set of Reeb chords $\Gamma\to \Gamma$.  We will use the following terminology for disks below.
\begin{itemize}
\item A holomorphic disk $u\colon D\to X_{\Lambda}(\epsilon)$ is an {\em $(a;c_1\dots c_m)$-chord-chord disk} if it has the following properties:
\begin{itemize}
\item $u(\pa D)\subset L\cup C$.
\item $u$ has one positive boundary puncture mapping to the Reeb chord $a$ of $\Gamma$.
\item $u$ has two Lagrangian intersection punctures mapping to $L\cap C$.
\item $u$ has negative boundary punctures mapping to the Reeb chords $c_1,\dots,c_m$ of $\Lambda$. (Note that there is an induced order of the Reeb chords at the negative boundary punctures.)
\end{itemize}
\item A holomorphic disk $u\colon D\to X_{\Lambda}(\epsilon)$ is an {\em $(\alpha;c_1\dots c_m)$-orbit-chord disk} if it has the following properties:
\begin{itemize}
\item $u(\pa D)\subset L$.
\item $u$ has one positive interior puncture with asymptotic marker mapping to the Reeb orbit $\alpha$ with a marker on the underlying geometric orbit in $Y_{\Lambda}(\epsilon)$.
\item $u$ has negative boundary punctures mapping to the Reeb chords $c_1,\dots,c_m$ of $\Lambda$. 
\end{itemize}
\item A holomorphic disk $u\colon D\to \R\times Y_\Lambda(\epsilon)$ is a {\em $(a_1,a_2;b)$-chord disk} if it has the following properties:
\begin{itemize}
\item $u(\pa D)\subset \Gamma\times\R$.
\item $u$ has two positive boundary punctures mapping to the Reeb chords $\Gamma_{2}\to\Gamma_{1}$ and $\Gamma_{1}\to\Gamma_{0}$ corresponding to the Reeb chords $a_1,a_2$ of $\Gamma$
\item $u$ has one negative boundary puncture mapping to the Reeb chord $b$ of $\Gamma$. 
\end{itemize}
\item A holomorphic disk $u\colon D\to \R\times Y_{\Lambda}(\epsilon)$ is an {\em $(a,\alpha)$-chord-orbit disk} if it has the following properties:
\begin{itemize}
\item $u(\pa D)\subset \Gamma\times\R$.
\item $u$ has one positive boundary puncture mapping to the Reeb chord $a$ of $\Gamma$.
\item $u$ has one negative interior puncture mapping to the Reeb orbit $\alpha$ in $Y_{\Lambda}(\epsilon)$.
\end{itemize}
Representing the domain disk as the unit disk with the positive boundary puncture at $1$ and the interior puncture at the origin, there is an induced positive marker, the image of the positive real axis on $\alpha$.  
\end{itemize}

\subsection{Uniqueness of the $(c';c)$-chord-chord disk $S_c$}\label{s:unique}
Let $c$ be a Reeb chord of $\Lambda\subset Y$. Consider the neighborhood $U(c)\subset Y$ of $c$ with coordinates $(\xi^{\pm},\eta^{\pm},\zeta^{\pm})\in\C^{n-1}\times\R$. By definition of these coordinates they are related to the coordinates $(x_2,y_2,z)$ in the handle, see \eqref{e:handlecoord}, where they overlap as follows
\begin{align*}
z &=\zeta^{\pm} + \xi^{\pm}\cdot \eta^{\pm},\\
x_2&=-\eta^{\pm},\\
y_2&=\xi^{\pm}.
\end{align*}

Consider next the region $A(c)$ as in \eqref{e:dfnA(c)}. Topologically, $A(c)$ is homeomorphic to $S^{1}\times \R\times D^{2n-2}$, where $S^{1}$ corresponds to the Reeb orbit corresponding to the Reeb chord $c$. Consider a point $p$ in $A(c)$. We want to map $p$ into $\C^{n-1}$. To this end we may use either the projection $\pi^{-}$ or the projection $\pi^{+}$ corresponding to coordinates $(\xi^{-},\eta^{-},\zeta^{-})$ and $(\xi^{+},\eta^{+},\zeta^{+})$ and then using the change of coordinates to get to $\C^{n-1}\subset \C^{n}$ with coordinates $(x_2,y_2)$. We think of $\pi^{-}$ ($\pi^{+}$) as the result of translating $p$ backwards (forwards) along $c$ to $c^{-}$ (to $c^{+}$) and then projecting to $\C^{n-1}$. From the change of coordinates relating $(\xi^{-},\eta^{-},\zeta^{-})$ to $(\xi^{+},\eta^{+},\zeta^{+})$ we get
\begin{equation}\label{e:monodromy1}
\pi^{+}=J_0\circ \pi^{-}.
\end{equation}

Let $\pr\colon\tilde A(c)\to A(c)$ denote the $4$-fold cover of $A(c)$. Then we define a projection
\begin{equation}\label{e:dfnprojCn-1}
\pi\colon\tilde A(c)\to\C^{n-1}
\end{equation}
as follows. Pick a base point $\tilde p_0\in \tilde A(c)$. If $\tilde p\in\tilde A(c)$ then connect $\tilde p$ to $\tilde p_0$ with a path $\tilde\gamma$ and let $\pi(\tilde p)\in\C^{n-1}$ be the result of translating $\pr(\tilde p)$ along the path $\gamma=\pr(\tilde\gamma)$ and then projecting to $\C^{n-1}$. As mentioned above, each time we go around $S^{1}\subset A(c)$ the projection to $\C^{n-1}$ changes by multiplication with $J_{0}$ and it follows that $\pi\colon\tilde A(c)\to\C^{n-1}$ is well defined.

Let $u\colon D\to A(c)$ be a map from a Riemann surface. Lifting $u$ we get a map $\tilde u\colon \tilde D\to\tilde A(c)$, where $\tilde D$ is a (not necessarily non-trivial) $4$-fold cover of $D$.

Let $\tilde u^{\perp}=\pi\circ \tilde u\colon \tilde D\to \C^{n-1}$ and note that the complex structure on $D$ induces a complex structure of $\tilde D$.
\begin{lma}\label{l:tildemap}
If $u\colon D\to A(c)$ is $J_c$-holomorphic then $\tilde u^{\perp}\colon \tilde D\to\C^{n-1}$ is $J_0$-holomorphic.
\end{lma}

\begin{pf}
Locally, the projection to $\C^{n-1}$ composed with $u$ is $J_{0}$-holomorphic by Lemma \ref{l:lochol}. The result follows. (The role of the covering construction is merely to piece all possible local holomorphic maps together to a map into $\C^{n-1}$.)
\end{pf}

As shown in Lemma \ref{l:S(c)}, there is an obvious $J_{c}$-holomorphic $(c';c)$-chord-chord disk $S_c$ which agrees with the Reeb chord strips of $c'$ and of $c$ in the upper- and lower symplectization ends respectively and which consists of a part of the $(x_1,y_1)$-plane inside the handle. We will show next that the disk $S_{c}$ is the only $(c';c)$-chord-chord disk in $A(\delta\epsilon)$ for some $0<\delta<1$ and all sufficiently small $\epsilon$. To this end we discuss how to parameterize maps with image in a neighborhood of $S_c$.

The neighborhood is parameterized by the total space of a bundle over the space of conformal structures on the $4$-punctured disk the fiber of which is a direct sum $\sblv_{2,\delta}\oplus V_{\rm sol}$ where $\sblv_{2,\delta}$ is the weighted Sobolev space of vector fields along (a reference parametrization of) $S_{c}$) with two derivatives in $L^{2}$ and with positive exponential weights of the form $e^{\delta|\tau|}$, $0<\delta\ll 1$ in neighborhoods the ends which we think of as $[0,\infty)\times[0,1]$, and where $V_{\rm con}$ is spanned by two cut-off solutions near the punctures mapping to Reeb chords corresponding to translations in the symplectization ends. To get a parametrization we use the exponential map on vector fields in our standard coordinates on the neighborhood of $S_c$. We view the $\bar{\partial}_{J}$-operator as a section of the corresponding bundle of complex anti-linear and we say that a solution of the $\bar{\partial}_{J}$-equation is transversely cut out if the linearization of this operator is surjective at the solution. For more details on this construction we refer to \cite[Appendix B.4]{EkholmJEMS} and \cite[Section 3]{EESTransactions}.

\begin{lma}\label{l:locunique}
The holomorphic disk $S_{c}\subset \hat X_{\Lambda}(\epsilon;c)$ is uniformly transversely cut out. Consequently, there exists $0<\delta<1$ such that for all sufficiently small $\epsilon>0$, $S_c$ is the only $(c';c)$-chord-chord disk in $A(c;\delta\epsilon)\subset \hat X_{\Lambda}(c;\epsilon)$.
\end{lma}

\begin{pf}
Consider the linearized $\bar\pa_{J_c}$-operator at $S_{c}$ acting on the space of vector fields with two derivatives in $L^{2}$ weighted by small positive exponential weights at the ends and augmented by cut-off vector fields corresponding to translations. This linearized operator is invertible (since its boundary condition is of a standard form and the linearized Reeb flow at the Reeb chords takes tangent spaces of the corresponding Legendrian submanifolds to transverse subspaces). Restricting to tangent vectors of $2$-norm $\le C\epsilon^{q}$ where $C$ is chosen so that the exponential map maps these variations inside the $\epsilon^{q}$-neighborhood where $\hat V_{\pm\epsilon}$ is a product we find that $0<\delta<1$ exists as claimed.
\end{pf}

\begin{rmk}
The reason for restricting to tangent vectors mapping into an $\epsilon^{q}$-neighborhood of $S_c$ is that if longer tangent vectors are included then we must take into account the change in the Reeb flow outside $W(\epsilon^{q})$. This change leads to a blow up in the variation of the derivative of the $\bar\pa_{J_{c}}$-operator and the implicit function theorem cannot be used to guarantee uniqueness as $\epsilon\to 0$.
\end{rmk}

Our next lemma shows that $S_c$ is unique over all.

\begin{lma}\label{l:unique}
The $J_c$-holomorphic disk $S_c$ is the only $(c;c')$-chord-chord in $\hat X_{\Lambda}(c;\epsilon)$.
\end{lma}

\begin{pf}
We assume that there exists a sequence $\{\epsilon_j\}_{j=1}^{\infty}$ such that $\epsilon_j\to 0$ as $j\to\infty$ and such that for each $j$ there is a  $(c';c)$-chord-chord disk $S^{j}_{c}$ in $\hat X_{\Lambda}(c;\epsilon)$ which does not lie in $A(c,\epsilon_j)$.

We will use the symplectic forms $e^{\sigma_j t}\lambda$ on the ends of $\hat X_{\Lambda}(c;\epsilon_j)$, where $\sigma_j>0$ is a small constant. More precisely, there exists $t_0>0$ such that $S_{c}^{j}$ lies inside a $\delta$-neighborhood of $c'\times\R$ and $c\times \R$ where $\delta\ll\epsilon^{p}$ in the region $t>|t_0|$. Take $\sigma>0$ such that $\sigma t_0\ll \epsilon^{2p+1}$.

Let $v_j\colon D\to \hat X(c;\epsilon_{j})$ be the map corresponding to $S^{j}_{c}$. Let
\[
\hat X_{\Lambda;t_0}(c;\epsilon_j)=
\hat X_{\Lambda}(c;\epsilon_j)-\bigl(Y\times(-\infty,-t_0)\cup Y_\Lambda\times(t_0,\infty)\bigr).
\]
Let
\[
E=v^{-1}_j(A(c;\epsilon)\cap\hat X_{\Lambda;t_0}(c;\epsilon))
\]
and consider the restriction
\[
v_j\colon E\to A(c;\epsilon)\cap \hat X_{\Lambda;t_0}(c;\epsilon).
\]
The boundary of $E$ decomposes as
\[
\pa E=\pa_{\pm} E\,\,\cup\,\,\pa_{L\cup C} E\,\,\cup\,\, \pa_0 E,
\]
where $v_{j}(\pa_+ E)$ ($v_{j}(\pa_- E)$) is a curve near the Reeb chord $c'$ ($c$) in the symplectization level at $t=t_0$ (at $t=-t_0$), where $\pa_{L\cup C} E\subset \pa D$ and $v_{j}(\pa_{L\cup C}E)\subset L\cup C$, and where $\pa_0 E=v_j^{-1}\bigl(\pa A(c;\epsilon)\bigr)$.

With notation as in Lemma \ref{l:tildemap}, consider the maps $\tilde v_j\colon \tilde E\to\tilde A(c;\epsilon)$ and the corresponding map $\tilde v^{\perp}_j\colon \tilde E\to\C^{n-1}$. Write
\[
\pa\tilde E=\pa_{\pm}\tilde E\,\,\cup\,\, \pa_{L\cup C}\tilde E\,\,\cup\,\,\pa_{0}\tilde E
\]
where $\pa_{\pm}\tilde E$ is the preimage of $\pa_{\pm} E$ etc, under the $4$-fold covering map.

Let $B^{n-1}(r)$ denote the ball of radius $r$ in $\C^{n-1}$ centered at the origin. Note that $\tilde v^{\perp}(\pa_{\pm}\tilde E)\subset B^{n-1}(\delta)\subset B^{n-1}(\epsilon^{2p}_{j})$ and that $\tilde v^{\perp}(\pa_{L\cup C}\tilde E)\subset \R^{n-1}\cup i\R^{n-1}$. Consequently, if $\pa_{0}\tilde E\ne\emptyset$ we can find a ball $\tilde B\subset B^{n-1}(\epsilon_{j}^{q})$ of radius $\tfrac14\epsilon^{q}_{j}$ such that $\tilde v^{\perp}_j(\tilde E)$ passes through the center of $\tilde B$ and the boundary of $\tilde v^{\perp}_j(\tilde E)$ lies in $\pa B^{n-1}(\epsilon^{q}_{j})\cup \bigl(\R^{n}\cup i\R^{n}\bigr)\cap B^{n-1}(\epsilon^{q}_{j})$. It then follows by monotonicity that
\[
\area(\tilde v)\ge\area(\tilde v^{\perp})\ge c_0\epsilon^{2q},
\]
for some constant $c_0$.
In particular, with $\beta$ denoting a primitive of the symplectic form $\omega$ on $\hat X_{\Lambda;t_0}(c;\epsilon)$ we find that
\begin{equation}\label{e:wrongarea}
\int_E v^{\ast} (d\beta)\ge \tfrac14 c_0\epsilon^{2q}.
\end{equation}

On the other hand the area of $S_{c}^{j}\cap \hat X_{\Lambda;t_0}(c;\epsilon)$ is up to terms of size $\Ordo(\epsilon^{2p+1})$ estimated by
\[
\int_{\epsilon^{s+1}}^{\epsilon^{s}}
\left(\sqrt{\epsilon^{2p}+2x_1^{2}}-\sqrt{2}x_1\right)dx_1=\Ordo(\epsilon^{2p}),
\]
since the boundary $\pa_{\pm} E$ lies at distance $\Ordo(\delta)=\Ordo(\epsilon^{2p+1})$ from the Reeb chords. Thus,
\[
\int_E v^{\ast}(d\beta)=\Ordo(\epsilon^{2p}).
\]
This contradicts \eqref{e:wrongarea} and we find  that $\pa_{0}E=\emptyset$ and hence that $S^{j}_{c}(\epsilon_{j})$ lies entirely inside $A(c;\epsilon_j)$ for all $j$ large enough.

Thus for $\epsilon>0$ sufficiently small any $(c';c)$-chord-chord disk $S_{c}'$ in $\hat X_{\Lambda}(c;\epsilon)$ lies in $A(c;\epsilon)$. We show that $S'_{c}=S_{c}$ using an argument similar to the above: the integral of a primitive of the symplectic form in $\C^{n-1}$ over the boundary of $S_{c}$ equals $0$. Hence the corresponding integral equals $\Ordo(\delta)=\Ordo(\epsilon^{2p})$ for $S'_{c}$. Monotonicity then shows that the $\C^{n-1}$-component of $S'_{c}$ is of size $\Ordo(\epsilon^{p})$ and a standard bootstrap argument implies that $S'_{c}$ is at distance $\Ordo(\epsilon^{p})$ from $S_{c}$ with respect to the functional analytic norm used in Lemma \ref{l:locunique}. Lemma \ref{l:locunique} then implies $S_{c}'=S_{c}$.
\end{pf}

\subsection{Counting disks}
The purpose of this section is to prove the main result for holomorphic disks interpolating between Reeb chords of $\Gamma$ and Reeb orbits in $Y_{\Lambda}(\epsilon)$ and corresponding words of chords of $\Lambda$ and corresponding cyclic words, respectively. We begin by stating this main result. 

Fix $\act_0$ not in the action set of $\Lambda\subset Y$ and fix $\epsilon_0>0$ so that Lemmas \ref{l:neworbits} and \ref{l:newchords} hold for Reeb orbits in $Y_{\Lambda}(\epsilon)$ respectively Reeb chords of $\Gamma\subset Y_{\Lambda}(\epsilon)$ of action $<\act_0$ for $0<\epsilon<\epsilon_0$. Let $c_1\dots c_m$ be a word of Reeb chords of $\Lambda$ with $\act(c_1\dots c_m)<\act_0$. Let $\alpha$ denote the Reeb orbit in $Y_\Lambda(\epsilon)$ corresponding to the cyclic word $(c_1\dots c_m)^{\circ}$ with a marker on the underlying geometric orbit, let $a$ denote the Reeb chord of $\Gamma\subset Y_{\Lambda}(\epsilon)$ corresponding to the word $c_1\dots c_m$, and let $a_1=c_1\dots c_j$ and $a_2=c_{j+1}\dots c_m$ be the Reeb chords of $\Gamma$ corresponding to consecutive sub-words of $c_1\dots c_m$

\begin{thm}\label{l:glu}
For all sufficiently small $\epsilon>0$ the following holds.
\begin{itemize}
\item[{\rm (a)}]
The algebraic number of $(a;c_1\dots c_m)$-chord-chord disks in $X_{\Lambda}(\epsilon)$ equals $\pm 1$.
\item[{\rm (b)}]
The algebraic number of $(\alpha;c_1\dots c_m)$-orbit-chord disks in $X_{\Lambda}(\epsilon)$ equals $\pm 1$.
\item[{\rm (c)}]
The algebraic number of $(a_1,a_2;a)$-chord disks in $Y_\Lambda(\epsilon)\times\R$ equals $\pm 1$.
\item[{\rm (d)}]
The algebraic number of $(a,\alpha)$-chord-orbit disks in $Y_{\Lambda}(\epsilon)\times\R$ equals $\pm 1$.
\end{itemize}
\end{thm}

We prove Theorem \ref{l:glu} in a sequence of lemmas below.

\begin{lma}\label{l:onesimple}
Let $c$ be any Reeb chord of $\Lambda\subset Y$ with $\act(c)<\act_0$ and let $b$ be any Reeb chord of $\Lambda\subset Y$ with $\act(b)\le \act(c)$ then for all sufficiently small $\epsilon>0$ the algebraic number of $(b';b)$-chord disks in $X_{\Lambda}(c;\epsilon)$ equals $\pm 1$.
\end{lma}

\begin{pf}
Consider a $1$-parameter family of cobordisms $X^{t}_{\Lambda}(\epsilon)$, $0\le t\le 1$ obtained by deforming the handle and the identification map taking $\Lambda$ to the core sphere and which has the following properties:
\begin{itemize}
\item[(a)] $X^{0}_{\Lambda}(\epsilon)=\hat X_{\Lambda}(b;\epsilon)$,
\item[(b)] $X^{1}_{\Lambda}(\epsilon)=X_{\Lambda}(c;\epsilon)$, and
\item[(c)] $\epsilon\ll\Delta\act$ where $\Delta\act$ is the minimal difference $\act(\alpha)-\act(\beta)$, where $\alpha$ and $\beta$ are Reeb orbits in $Y$ or cyclic words or Reeb chords of $\Lambda\subset Y$ and where $\alpha\ne\beta$.
\end{itemize}
Also fix a $1$-parameter family of almost complex structures $J^t$ on $X^{t}_{\Lambda}(\epsilon)$ such that $J^{0}=J_{b}$ and such that if $J=J^{1}$ then the moduli space of $J$-holomorphic $(b';b)$-chord-chord disks in $X(c;\epsilon)$ is a transversely cut out  compact $0$-manifold. (Transversality is straightforward, we can perturb $J$ near the positive puncture.)  Let $\M^{t}$ denote the moduli space of $J^{t}$-holomorphic $(b';b)$-chord-chord disks in $X^{t}_{\Lambda}(\epsilon)$ and consider the parameterized moduli space
\[
\M^{[0,1]}=\bigcup_{0\le t\le 1}\M^{t}.
\]
After small perturbation fixed near $t=0,1$, $\M^{[0,1]}$ becomes a compact oriented $1$-manifold with a natural compactification consisting of several level disks and of disks on the boundary. In fact we claim that there are no broken disks so that $\M^{[0,1]}$ is a compact $1$-manifold with boundary
\begin{equation}\label{e:boundary}
\pa\M^{[0,1]}=\M^{1}\cup -\M^{0}.
\end{equation}
To see this, assume that there is some several level disk in the compactification of $\M^{[0,1]}$. By exactness of the Lagrangian boundary condition the only splitting which can occur corresponds to splitting at a Reeb chord at positive or negative infinity or at a Reeb orbit at negative infinity. The existence of such a splitting however contradicts $\act(c')-\act(c)=\Ordo(\epsilon^{2p})$ by Stokes' theorem and (c) above.

The lemma now follows from \eqref{e:boundary}, (a), the properties of $J^{0}$ and $J^{1}$, and Lemma \ref{l:unique}.
\end{pf}

\begin{pf}[Proof of Lemma \ref{l:glu}]
We use notation as in the formulation. Assume inductively that the number $n(a;c_1\dots c_m)$ of $(a;c_1\dots c_m)$-chord-chord disks equals $\pm 1$ for all Reeb chord words of length $m$. Lemma \ref{l:onesimple} implies that this holds for $m=1$.

Consider the moduli space of disks with two positive punctures at Reeb chords $a_1$ connecting $\Gamma_{2}$ to $\Gamma_{1}$ and $a_2$ connecting $\Gamma_{1}$ to $\Gamma_{0}$, corresponding to Reeb chords of  $\Gamma\subset Y_{\Lambda}(\epsilon)$ where $a_1$ corresponds to the Reeb chord word $c_1\dots c_l$ and $a_2$ to the word $c_{l+1}\dots c_{m+1}$, with negative punctures at Reeb chords $c_1,\dots,c_{m+1}$ of $\Lambda\subset Y$, and with two Lagrangian intersection punctures. This moduli space is $1$-dimensional. Transversality by perturbing the almost complex structure is again straightforward, the disk have injective points near the positive punctures since we use slightly distinct Lagrangians. Then, since $\epsilon\ll\act(a_1)+\act(a_2)-\sum_{j=1}^{m+1} \act(c_j)$ it follows as in the proof of Lemma \ref{l:onesimple} that the boundary of this moduli space consists of the following configurations:
\begin{itemize}
\item two level disks with top-level a $(a_1,a_2;a)$-chord disk and bottom level an $(a;c_1\dots c_{m+1})$-chord-chord disk, and
\item two component broken disks with components an $(a_1;c_1\dots c_{l})$-chord-chord disk and an $(a_2;c_{l+1}\dots c_{m+1})$-chord-chord disk joined at $C\cap L$.
\end{itemize}
Let $n(a_1,a_2;a)$ denote the numbers of $(a_1,a_2;a)$-chord disks. Then counting boundary components of the $1$-dimensional moduli space we find
\[
n(a_1,a_2;a)n(a;c_1\dots c_{m+1})=n(a_1;c_1\dots c_{l})n(a_2;c_{l+1}\dots c_{m+1}).
\]
However, by our inductive assumption $n(a_1;c_1\dots c_l)n(a_2;c_{l+1}\dots c_{m+1})=\pm 1$ and we conclude $n(a_1,a_2;a)=\pm 1$ and $n(a;c_1\dots c_{m+1})=\pm 1$. This proves (a) and (c).

For (b) and (d) we consider instead the $1$-dimensional moduli space of punctured cylinders with a positive puncture on one boundary component mapping to $C$ where the map is asymptotic to $a$ corresponding to the word $c_1\dots c_m$ and negative punctures at the other boundary components mapping to $L$ where the map is asymptotic to $c_1,\dots,c_{m}$. Representing the cylinder as $[-R,R]\times S^{1}$ we require that the positive puncture corresponds to $R\times 1$. The action argument above shows that the boundary of this $1$-dimensional moduli space consists of the following configurations:
\begin{itemize}
\item an $(a_1;c_1\dots c_l)$-chord-chord disk joined with a $(a_2;c_{l+1}\dots c_m)$-chord-chord disk at $L\cap C$ and
\item an $(a;\alpha)$-chord-orbit disk joined to an $(\alpha;c_1\dots c_m)$-orbit-chord disk. Here the boundary puncture of the $(a;\alpha)$-chord-orbit disk induces a marker on $\alpha$, which is then the marker at the positive puncture of the $(\alpha;c_1\dots c_m)$-orbit-chord disk.
\end{itemize}
Since the algebraic number of $(a,c_1\dots c_k)$-disks equals $\pm 1$ a repetition of the above count of ends of moduli spaces shows that (a) implies (b) and (d) with the induced marker. To see that (d) holds for any marker, consider the cobordism of moduli spaces corresponding to moving the marker on the orbit. Since there is no Reeb orbit or chord with action between $\alpha$ and the cyclic word there is no splitting in the moduli space and the count is independent of the location of the marker. This finishes the proof.
\end{pf}

\begin{proof}[Proof of Theorem \ref{t:count}]
Theorem \ref{t:count} is a restatement of Theorem \ref{l:glu}. 
\end{proof}

\subsection{Analogous results for mixed Reeb chords}
The analogues of the results for pure Reeb chords in Section \ref{s:unique} for mixed chords $a$ connecting an arbitrary Legendrian submanifold to the attaching sphere $\Lambda$ are similar to the pure chord case but simpler. We first introduce notation. Let $\Lambda^{0}$ be a Legendrian submanifold disjoint from $\Lambda$. 

\begin{itemize}
\item A holomorphic disk $u\colon D\to X_{\Lambda}(\epsilon)$ is an {\em $(a;a_1c_1\dots c_m)$-chord-chord disk} if it has the following properties:
\begin{itemize}
\item $u(\pa D)\subset L\cup C$.
\item $u$ has one positive boundary puncture mapping to the Reeb chord $a$ of $\Gamma$.
\item $u$ has one Lagrangian intersection punctures mapping to $L\cap C$.
\item $u$ has negative boundary punctures mapping to the Reeb chords $a_1,c_1,\dots,c_m$, where $a_1$ is a Reeb chord connecting $\Lambda_0$ to $\Lambda$ and $c_j$ connects $\Lambda$ to itself. (Note that there is an induced order of the Reeb chords at the negative boundary punctures.)
\end{itemize}
\item A holomorphic disk $u\colon D\to X_{\Lambda}(\epsilon)$ is an {\em $(b;c_1\dots c_m)$-chord-chord disk} if it has the following properties:
\begin{itemize}
\item $u(\pa D)\subset L$.
\item $u$ has one positive boundary puncture mapping to the Reeb chord $b$ in $Y_{\Lambda}(\epsilon)$ corresponding to $(a_1c_1\dots c_m a_2)$.
\item $u$ has negative boundary punctures mapping to the Reeb chords $a_1c_1,\dots,c_m a_2$ of $\Lambda$. 
\end{itemize}
\item A holomorphic disk $u\colon D\to Y_\Lambda(\epsilon)\times\R$ is a {\em $(a_1,a_2;b)$-chord disk} if it has the following properties:
\begin{itemize}
\item $u(\pa D)\subset \Gamma\times\R\cup\Lambda_0\times\R$.
\item $u$ has two positive boundary punctures mapping to the Reeb chords $a_1,a_2$ connecting $\Gamma$ to $\Lambda_0$ and $\Lambda_0$ to $\Gamma$, respectively.
\item $u$ has one negative boundary puncture mapping to the Reeb chord $b$ connecting of $\Lambda_0$. (Note that there is an induced order of the Reeb chords at the positive boundary punctures.)
\end{itemize}
\end{itemize}

In analogy with Lemma \ref{l:unique} we have the following result for basic disks: 

\begin{lma}\label{l:unique2}
For a mixed chord $a$ of between $\Lambda$ and $\Lambda_0$,
the $J_a$-holomorphic disk $S_c$ is the only $(a;a')$-chord-chord disk in $\hat X_{\Lambda}(a;\epsilon)$.
\end{lma}

\begin{proof}
Analogous to Theorem \ref{l:unique}.
\end{proof}

With this result established the following main result for disks connecting the above type of chords are proved using the exact same gluing and compactness arguments as in the proof of Theorem \ref{l:glu}.

\begin{thm}\label{l:gluLeg}
For all sufficiently small $\epsilon>0$ the following holds.
\begin{itemize}
\item[{\rm (a)}]
The algebraic number of $(a;a_{1}c_1\dots c_m)$-chord-chord disks in $X(\epsilon)$ equals $\pm 1$.
\item[{\rm (b)}]
The algebraic number of $(b;a_{1}c_1\dots c_ma_{m})$-orbit-chord disks in $X(\epsilon)$ equals $\pm 1$.
\item[{\rm (c)}]
The algebraic number of $(a_1,a_2;b)$-chord disks in $Y_\Lambda(\epsilon)\times\R$ equals $\pm 1$.
\end{itemize}
\end{thm}

\begin{proof}
Analogous to Theorem \ref{l:unique}. 
\end{proof}

\bibliographystyle{plain}
\bibliography{beerefs}

\end{document}